\documentclass{amsart}

\makeatletter \thm@headfont{\bfseries\scshape} \makeatother
\usepackage[english]{babel}
\usepackage[utf8]{inputenc}
\usepackage{latexsym,euscript}
\usepackage{amsbsy,amscd,amsfonts,amsmath,amsopn,amssymb,amstext,amsthm,amsxtra}
\usepackage{epsf,epsfig,graphics,color,verbatim}
\usepackage{cite,float,multirow}
\usepackage{graphicx}
\usepackage{subfigure}
\usepackage{enumerate}
\usepackage{mathtools}
\usepackage{wasysym}
\usepackage{tikz}
\usepackage{hyperref}

\theoremstyle{plain}
\newtheorem{theorem}{Theorem}[section]
\newtheorem{lemma}[theorem]{Lemma}

\newtheorem{proposition}[theorem]{Proposition}

\theoremstyle{definition} 
\newtheorem{definition}[theorem]{Definition}
\newtheorem{remark}[theorem]{Remark}

\newcommand{\dd}{\textup{d}}

\newcommand{\NN}{\mathbb{N}}

\newcommand{\xbf}{\mathbf{x}}

\newcommand{\set}[1]{\left\{ #1 \right\}}
\newcommand{\setb}[1]{\left( #1 \right)}
\newcommand{\abs}[1]{\left| #1 \right|}

\newcommand{\ind}[1]{\mathbf{1}_{ #1 }}
\newcommand{\symb}{\mathrm{symb}}

\newcommand{\by}[1]{\textsc{#1}:}

\allowdisplaybreaks

\begin{document}
	
	\title[Average diameter of blocks of uniformly distributed sequences]{On the asymptotic average 
	diameter of blocks of uniformly distributed sequences and related results}
	
	\author[1]{Sebastian Heintze}
	\address{\textbf{Sebastian Heintze}\\
		Department of Artificial Intelligence and Human Interfaces\\
		Faculty of Digital and Analytical Sciences\\
		University of Salzburg\\
		Hellbrunnerstraße 34\\
		5020 Salzburg\\
		Austria}
	\email{sebastian.heintze@plus.ac.at}
	
	\author[2]{Wolfgang Trutschnig}
	\address{\textbf{Wolfgang Trutschnig}\\
		Department of Artificial Intelligence and Human Interfaces\\
		Faculty of Digital and Analytical Sciences\\
		University of Salzburg\\
		Hellbrunnerstraße 34\\
		5020 Salzburg\\
		Austria}
	\email{wolfgang@trutschnig.net}
	
	\def\shortauthors{S. Heintze --- W. Trutschnig}
	
	
	\keywords{Uniformly distributed sequences, copula, $ d $-stochastic measure}
	\thanks{The authors gratefully acknowledge the support of the project `IDA Lab Salzburg' (20102/F2300464-KZP; 20204-WISS/225/348/3-2023).}
	
	
	\begin{abstract}
		This paper was triggered by recent results on the maximal `average distance between consecutive points' of uniformly distributed sequences (u.f.d.s.).
		Here we address a generalized version of this question, consider pairwise maximal/minimal/total distances in blocks/segments of $ d \geq 2 $ consecutive points of u.f.d.s., and derive sharp upper bounds for all three aggregations.
		Our main idea of proof consists in, firstly, adding degrees of freedom, secondly, translating the resulting problem to a solvable optimization problem over the compact family of $ d $-stochastic measures, and, thirdly, showing that the obtained bounds are also sharp bounds for the original problem.
	\end{abstract}
	
	
	\maketitle
	
	
	\par
	
	\section{Introduction}
	
	A sequence $ (x_n)_{n \in \NN} $ in the unit interval $ [0,1] $ is called uniformly distributed if for every interval $ [a,b] \subseteq [0,1 ]$ we have that ($ \ind{[a,b]} $ denoting the indicator function of $ [a,b] $)
	\begin{equation*}
		\lim\limits_{n \to \infty} \frac{1}{n} \sum_{i=1}^{n} \ind{[a,b]}(x_i) = b-a = \lambda([a,b])
	\end{equation*}
	holds.
	Translating to the stochastic perspective, $ (x_n)_{n \in \NN} $ is uniformly distributed if and only if the empirical measure $ \vartheta_n $, given by ($ \delta_a $ denoting the Dirac measure at $ a $)
	\begin{equation*}
		\vartheta_n = \frac{1}{n} \sum_{i=1}^{n} \delta_{x_i} \in \mathcal{P}([0,1]),
	\end{equation*}
	converges weakly to the Lebesgue measure $ \lambda $ for $ n \to \infty $.
	In what follows, we will let $ \mathcal{U} $ denote the family of all uniformly distributed sequences (u.f.d.s., for short).
	For extensive background on u.f.d.s.\ we refer to the textbooks \cite{Ku-Ni,Dr-Ti} and the references therein.
	
	Our contribution is motivated by the following question tackled by Pillichshammer and Steinerberger in \cite{Pi-St}, where the authors studied, how large the asymptotic `average distance between consecutive points' of $ (x_n)_{n \in \NN} \in \mathcal{U} $ can possibly be.
	Among other things, they proved that
	\begin{equation}
		\label{eq:pill}
		\sup_{(x_n)_{n \in \NN} \in \mathcal{U}} \limsup\limits_{n \to \infty} \frac{1}{n} \sum_{i=1}^n \abs{x_{i+1}-x_i} \leq \frac{1}{2}
	\end{equation}  
	holds and that the upper bound $ \frac{1}{2} $ is attainable, i.e., that there exist u.f.d.s.\ such that equation \eqref{eq:pill} becomes an identity (one example being the van der Corput sequence in base $ 2 $).
	Here we address a generalized version of this problem, consider pairwise maximal/minimal/total distances in blocks/segments of $ d \geq 2 $ consecutive points, and work with the three (continuous) functions $ f_{\min}, f_{\max}, f_{\sum} \colon [0,1]^d \to [0,1] $, defined by
	\begin{align}
		\label{eq:f-min}
		f_{\min}(\xbf) &\coloneqq \min_{i,j \in [d],\ i \neq j} \abs{x_i-x_j}, \\
		\label{eq:f-max}
		f_{\max}(\xbf) &\coloneqq \max_{i,j \in [d],\ i \neq j} \abs{x_i-x_j}, \\
		\label{eq:f-sum}
		f_{\sum}(\xbf) &\coloneqq \sum_{i,j \in [d]} \abs{x_i-x_j},
	\end{align}
	whereby $ \xbf = (x_1,\ldots,x_d) $ and $ [d] \coloneqq \set{1,\ldots,d} $.
	Notice that for $ d=2 $, up to the factor $ 2 $ all three functions boil down to the pairwise distance of consecutive points as considered in inequality \eqref{eq:pill}. 
	Our main objective is to determine
	\begin{equation*}
		B_f^d \coloneqq \sup_{(x_n)_{n \in \NN} \in \mathcal{U}} \limsup\limits_{n \to \infty} \frac{1}{n} \sum_{i=1}^{n} f(x_i,x_{i+1},\ldots,x_{i+d-1})
	\end{equation*}
	for each of the three functions $ f_{\min}, f_{\max}, f_{\sum} $.
	In fact, we will prove that
	\begin{equation*}
		B_{f_{\min}}^d = \frac{1}{d}, \quad
		B_{f_{\max}}^d = \frac{d-1}{d}, \quad
		B_{f_{\sum}}^d = \frac{d^2-1}{3}
	\end{equation*}
	for every $ d \geq 2 $.
	Doing so, our proof strategy can be subsumed as follows:
	\begin{enumerate}[i)]
		\item
		Considering that for every u.f.d.s.\ $ (x_n)_{n \in \NN} $ and every $ j \in \NN_0 $ the shifted sequence and $ (x_{j+n})_{n \in \NN} $ is a u.f.d.s.\ too, we obviously have that the upper bound $ B_f^d $ fulfills
		\begin{equation*}
			B_f^d \leq \sup_{\xbf^1,\xbf^2,\ldots,\xbf^d \in \mathcal{U}} \limsup\limits_{n \to \infty} \frac{1}{n} \sum_{i=1}^{n} f(x_i^1,x_i^2,\ldots,x_i^d) \eqqcolon C_f^d.
		\end{equation*}
		\item
		We show that $ C_f^d $ can be expressed as
		\begin{equation}
			\label{eq:via-ds}
			C_f^d = \sup_{\vartheta \in \mathcal{P}_\textrm{ds}([0,1]^d)} \int_{[0,1]^d} f \, \dd\vartheta,
		\end{equation}
		i.e., the maximization problem can be translated to a problem of maximizing integrals over the family $ \mathcal{P}_\textrm{ds}([0,1]^d) $ of all $ d $-stochastic measures (see Section~\ref{sec:nota-prelim}).
		\item
		For determining the right-hand side of \eqref{eq:via-ds} it suffices to work with (handy) subfamilies of $ \mathcal{P}_\textrm{ds}([0,1]^d) $ which are dense (with respect to the topology induced by weak convergence).
		\item
		We show (the surprising fact) that $ B_f^d = C_f^d $ holds.
	\end{enumerate}
	The remainder of this paper is organized as follows:
	Section~\ref{sec:nota-prelim} gathers notation and relevant background on $ d $-stochastic measures and copulas that will be used in the sequel.
	Section~\ref{sec:trans-cop} justifies the afore-mentioned translation to the $ d $-stochastic setting, proves equation \eqref{eq:via-ds} as main result, and then establishes denseness for the family of all so-called shuffles of $ M $.
	Main objective of Section~\ref{sec:bounds-d-stoch} is to determine $ C_f^d $ for all three functions  $ f_{\min}, f_{\max}, f_{\sum} $.
	Finally, Section~\ref{sec:back-orig-prob} proves $ B_f^d = C_f^d $ by constructing u.f.d.s.\ either attaining the upper bound $ B_f^d $ or approximating $ B_f^d $ arbitrarily well. 
	
	\section{Notation and preliminaries}
	\label{sec:nota-prelim}
	
	For every metric space $ (S,\rho) $ the Borel $ \sigma $-field on $ S $ will be denoted by
	$ \mathcal{B}(S) $.
	Throughout the whole paper $ d \geq 2 $ will denote the dimension, bold symbols will denote vectors or sequences and we will write either $ \xbf = (x_1,\ldots,x_d) $ or $ \xbf = (x_1,x_2,\ldots) $, respectively; from the context it will be clear, whether $ \xbf $ is a vector or a sequence.
	$ [N] $ will denote the set $ \set{1,\ldots,N} $ for every $ N \in \NN $, and $ \pi_i: [0,1]^d \to [0,1] $ the projection on the $ i $-th coordinate, i.e., $ \pi_i(x_1,\ldots,x_d) = x_i $.
	
	$ \mathcal{P}([0,1]^d) $ will denote the family of all probability measures on $ \mathcal{B}([0,1]^d) $, $ \vartheta^{\pi_i} \in \mathcal{P}([0,1]) $ the $ i $-th marginal of $ \vartheta $, i.e., the push-forward of $ \vartheta $ via $ \pi_i $.
	A measure $ \vartheta \in \mathcal{P}([0,1]^d) $ is called $ d $-stochastic, if all univariate marginals of $ \vartheta $ coincide with the Lebesgue measure $ \lambda $ on $ [0,1] $; $ \mathcal{P}_\textrm{ds}([0,1]^d) $ will denote the family of all $ d $-stochastic measures on $ [0,1]^d $.
	Moreover, $ \mathcal{C}_d $ denotes the family of all $ d $-dimensional copulas, i.e., the distribution functions of $ d $-stochastic measures restricted to $ [0,1]^d $.
	Considering
	\begin{equation*}
		C(x_1,\ldots,x_d) = \vartheta\left( \bigtimes_{i=1}^d [0,x_i] \right), \qquad (x_1,\ldots,x_d) \in [0,1]^d
	\end{equation*}
	establishes a one-to-one correspondence between $ \mathcal{C}_d $ and $ \mathcal{P}_\textrm{ds}([0,1]^d) $.
	It is well known (see, e.g., \cite{Bi,Ku-To-Me-Vr}) that endowing $ \mathcal{P}([0,1]^d) $ with the Hutchinson~/ Wasserstein~/ Kantorovic-Rubinstein metric $ \rho_W $, defined by
	\begin{equation*}
		\rho_W(\mu,\nu) = \sup \set{\abs{\int_{[0,1]^d} f \,\dd\mu - \int_{[0,1]^d} f \, \dd\nu} : \ \ f \colon [0,1]^d \to \mathbb{R} \text{ is } 1 \text{-Lipschitz}},
	\end{equation*}
	yields a compact metric space $ (\mathcal{P}([0,1]^d),\rho_W) $, which has $ \mathcal{P}_\textrm{ds}([0,1]^d) $ as closed (hence compact) convex subset.
	Moreover (again see \cite{Bi,Ku-To-Me-Vr}), $ \rho_W $ is a metrization of weak convergence in $ \mathcal{P}([0,1]^d) $, i.e., a sequence $ (\vartheta_n)_{n \in \NN} $ in $ \mathcal{P}([0,1]^d) $ converges weakly to $ \vartheta $ if and only if for every continuous function $ f \colon [0,1]^d \to \mathbb{R} $ we have that
	\begin{equation*}
		\lim\limits_{n \to \infty} \int_{[0,1]^d} f(\xbf) \, \dd\vartheta_n(\xbf) = \int_{[0,1]^d} f(\xbf) \, \dd\vartheta(\xbf).
	\end{equation*}
	In the sequel we will also use the fact that weak convergence in $ \mathcal{P}_\textrm{ds}([0,1]^d) $ is equivalent to uniform convergence of the corresponding copulas.
	In other words, letting $ C, C_1, C_2, \ldots $ denote the $ d $-dimensional copulas corresponding to the $ d $-stochastic measures $ \vartheta, \vartheta_1, \vartheta_2, \ldots $, the following two assertions are equivalent (see \cite{Du-Se}):
	\begin{enumerate}[i)]
		\item
		The sequence $ (\vartheta_n)_{n \in \NN} $ converges weakly to $ \vartheta $.
		\item
		$ \lim\limits_{n \to \infty} d_{\infty}(C_n,C) = \lim\limits_{n \to \infty} \max \set{\abs{C_n(\xbf) - C(\xbf)} : \xbf \in [0,1]^d} = 0 $.
	\end{enumerate}	 
	
	A \emph{Markov kernel} from $ \mathbb{R} $ to $ \mathcal{B}(\mathbb{R}^{d-1}) $ is a mapping $ K \colon \mathbb{R} \times \mathcal{B}(\mathbb{R}^{d-1}) \to [0,1] $ such that $ x \mapsto K(x,B) $ is measurable for every fixed $ B \in \mathcal{B}(\mathbb{R}^{d-1}) $ and $ B \mapsto K(x,B) $ is a probability measure for every fixed $ x \in \mathbb{R} $.
	Suppose that $ \mathbf{Y} \colon \Omega \to \mathbb{R}^{d-1} $ and $ X \colon \Omega \to \mathbb{R} $ are random vectors/variables on a joint probability space $ (\Omega, \mathcal{A}, \mathbb{P}) $.
	Then a Markov kernel $ K \colon \mathbb{R} \times \mathcal{B}(\mathbb{R}^{d-1}) \to [0,1] $ is called regular conditional distribution of $ \mathbf{Y} $ given $ X $ if for every $ B \in \mathcal{B}(\mathbb{R}^{d-1}) $
	\begin{equation*}
		K(X(\omega),B) = \mathbb{E}(\ind{B} \circ \mathbf{Y} \mid X)(\omega)
	\end{equation*}
	holds $ \mathbb{P} $-a.s..
	It is well known that for each random vector $ (X,\mathbf{Y}) $ a regular conditional distribution $ K(\cdot,\cdot) $ of $ \mathbf{Y} $ given $ X $ exists, that $ K(\cdot,\cdot) $ is unique $ \mathbb{P}^X $-a.s.\ (i.e., unique for $ \mathbb{P}^X $-almost all $ x \in \mathbb{R} $) and that $ K(\cdot,\cdot) $ only depends on the distribution $ \mathbb{P}^{(X,\mathbf{Y})} $ of the vector $ (X,\mathbf{Y}) $.
	Hence, if $ (X,\mathbf{Y}) $ has distribution $ \vartheta \in \mathcal{P}_\textrm{ds}([0,1]^d) $, then we will denote (a version of) the regular conditional distribution of $ \mathbf{Y} $ given $ X $ by $ K_{\vartheta}(\cdot,\cdot) $, view it directly as function from $ [0,1] \times \mathcal{B}([0,1]^{d-1}) $ to $ [0,1] $, and refer to $ K_{\vartheta}(\cdot,\cdot) $ simply as regular conditional distribution or Markov kernel of $ \vartheta $.
	Note that for every $ \vartheta \in \mathcal{P}_\textrm{ds}([0,1]^d) $, its conditional regular distribution $ K_{\vartheta}(\cdot,\cdot) $, and a Borel set $ F \in \mathcal{B}([0,1]^{d}) $ we have (with $ F_x = \set{\mathbf{y} \in [0,1]^{d-1} : (x,\mathbf{y}) \in F} $ denoting the $ x $-cut of $ F $ for every $ x \in [0,1] $)
	\begin{equation*}
		\int_{[0,1]} K_{\vartheta}(x,F_x) \, \dd\lambda(x) = \vartheta(F).
	\end{equation*}
	In particular, 
	\begin{equation*}
		\label{eq:markov-kernel}
		\int_{[0,1]} K_{\vartheta}(x,E) \, \dd\lambda(x) = \lambda(E_{i_0})
	\end{equation*}
	holds for $ E = \bigtimes_{i=1}^{d-1} E_i $ with $ E_i = [0,1] $ for all $ i \neq i_0 $.
	On the other hand, every Markov kernel $ K \colon [0,1] \times \mathcal{B}([0,1]^{d-1}) \to [0,1] $ fulfilling equation \eqref{eq:markov-kernel} is the regular conditional distribution of a $ d $-stochastic measure $ \vartheta \in \mathcal{P}_\textrm{ds}([0,1]^d) $.
	More generally, for every continuous or $ \vartheta $-integrable function $ f \colon [0,1]^d \to \mathbb{R} $ the following identity (to which we will loosely refer as `disintegration' in the sequel)
	holds:
	\begin{equation*}
		\int_{[0,1]^d} f(\xbf) \, \dd\vartheta(\xbf) = \int_{[0,1]} \int_{[0,1]^{d-1}} f(x_1,x_2,\ldots,x_d) K_{\vartheta}(x_1, \dd(x_2,\ldots,x_d)) \, \dd\lambda(x_1) 
	\end{equation*}
	For more details and properties of conditional expectation, regular conditional distributions 
	and disintegration we refer to \cite{Ka,Kl} as well as to \cite{Ba, Bh-Wa}.

	\section{Translating the maximization to the $d$-stochastic setting}
	\label{sec:trans-cop}
	
	The bivariate version of the following general result was already used in \cite{Pi-St}.
	For the sake of completeness and clarity we state the general $ d $-dimensional version and prove it using properties of weak convergence of measures and compactness ($ x_i^j $ denotes the $ i $-th element of the sequence $ \xbf^j $).
	
	\begin{lemma}
		Suppose that $ \xbf^1, \xbf^2, \ldots, \xbf^d $ are u.f.d.s.\ and that $ f \colon [0,1]^d \to \mathbb{R} $ is continuous. 
		Then for every accumulation point $ a $ of the sequence $ (a_n)_{n \in \NN} $, given by
		\begin{equation*}
			a_n \coloneqq \frac{1}{n} \sum_{i=1}^{n} f(x_i^1,x_i^2,\ldots,x_i^d),
		\end{equation*}
		there exists some $ d $-stochastic measure $ \vartheta \in \mathcal{P}_\textrm{ds}([0,1]^d) $ with
		\begin{equation*}
			a = \int_{[0,1]^d} f \, \dd\vartheta.
		\end{equation*}
	\end{lemma}
	\begin{proof}
		Continuity of $ f $ and compactness of $ [0,1]^d $ imply that all accumulation points of the sequence $ (a_n)_{n \in \NN} $ are contained in a bounded interval. 
		Suppose that
		\begin{equation*}
			a = \lim\limits_{j \to \infty} \frac{1}{n_j} \sum_{i=1}^{n_j} f(x_i^1,x_i^2,\ldots,x_i^d)
		\end{equation*}
		for some strictly increasing sequence $ (n_j)_{j \in \NN} $. 
		Let $ \delta_{x_i^1,x_i^2,\ldots,x_i^d} $ denote the Dirac measure at $ (x_i^1, x_i^2, \ldots, x_i^d) $ and set
		\begin{equation*}
			\vartheta_{n_j} \coloneqq \frac{1}{n_j} \sum_{i=1}^{n_j} \delta_{x_i^1,x_i^2,\ldots,x_i^d} \in \mathcal{P}([0,1]^d),
			\qquad j \in \NN.
		\end{equation*}
		Then $ (\vartheta_{n_j})_{j \in \NN} $ is a sequence in $ \mathcal{P}([0,1]^d) $, which by compactness has a weakly convergent subsequence $ (\vartheta_{n_{j_l}})_{l \in \NN} $ with weak limit $ \vartheta \in \mathcal{P}([0,1]^d) $.
		This directly yields
		\begin{align*}
			a &= \lim\limits_{l \to \infty} \frac{1}{n_{j_l}} \sum_{i=1}^{n_{j_l}} f(x_i^1,x_i^2,\ldots,x_i^d) \\
			&= \lim\limits_{l \to \infty} \int_{[0,1]^d} f(z_1,\ldots,z_d) \, \dd\vartheta_{n_{j_l}}(z_1,\ldots,z_d) \\
			&= \int_{[0,1]^d} f(z_1,\ldots,z_d) \, \dd\vartheta(z_1,\ldots,z_d),
		\end{align*} 
		so it remains to show that $ \vartheta $ is $ d $-stochastic.
		Applying Portmanteau's theorem (see \cite{Bi}) for the rectangles $ E_1 = [0,x] \times [0,1]^{d-1} $ and $ E_2 = [0,x+\Delta) \times [0,1]^{d-1} $ with $ x \in [0,1) $ and $ \Delta \in (0,1] $ fulfilling $ x+\Delta \in [0,1] $ yields
		\begin{align*}
			\vartheta(E_1) &\geq \limsup\limits_{l \to \infty} \vartheta_{n_{j_l}}([0,x] \times [0,1]^{d-1}) \\
			&= \limsup\limits_{l \to \infty} \frac{1}{n_{j_l}} \sum_{i=1}^{n_{j_l}} \delta_{x_i^1}([0,x]) = \lambda([0,x]) = x
		\end{align*}
		as well as
		\begin{align*}
			\vartheta(E_2) &\leq \liminf\limits_{l \to \infty} \vartheta_{n_{j_l}}([0,x+\Delta) \times [0,1]^{d-1}) \\
			&= \liminf\limits_{l \to \infty} \frac{1}{n_{j_l}} \sum_{i=1}^{n_{j_l}} \delta_{x_i^1}([0,x+\Delta)) \\
			&= \lambda([0,x+\Delta)) = x+\Delta.
		\end{align*}
		This implies
		\begin{equation*}
			x \leq \vartheta\left( [0,x] \times [0,1]^{d-1} \right) \leq \vartheta\left([0,x +\Delta) \times [0,1]^{d-1} \right) \leq x +\Delta,
		\end{equation*}
		so, using the fact that probability measures are continuous from above and considering $ \Delta \to 0 $, we conclude that $ x = \vartheta\left( [0,x] \times [0,1]^{d-1} \right) $.
		As $ x \in [0,1) $ was arbitrary, it follows that the first marginal $ \vartheta^{\pi_1} $ of $ \vartheta $ coincides with $ \lambda $.
		Since the previous arguments also work for any other coordinate, it altogether follows that $ \vartheta $ is $ d $-stochastic. 
	\end{proof}
	
	\begin{lemma}
		For every $ d \geq 2 $ the following identity holds:
		\begin{equation*}
			C_f^d = \sup_{\xbf^1,\xbf^2,\ldots,\xbf^d \in \mathcal{U}} \limsup\limits_{n \to \infty} \frac{1}{n} \sum_{i=1}^{n} f(x_i^1,x_i^2,\ldots,x_i^d) = \sup_{\vartheta \in \mathcal{P}_\textrm{ds}([0,1]^d)} \int_{[0,1]^d} f \, \dd\vartheta.
		\end{equation*}
	\end{lemma}	
	\begin{proof}
		As a direct consequence of the previous lemma we already have that
		\begin{equation*}
			\sup_{\xbf^1,\xbf^2,\ldots,\xbf^d \in \mathcal{U}} \limsup\limits_{n \to \infty} \frac{1}{n} \sum_{i=1}^{n} f(x_i^1,x_i^2,\ldots,x_i^d) \leq \sup_{\vartheta \in \mathcal{P}_\textrm{ds}([0,1]^d)} \int_{[0,1]^d} f \, \dd\vartheta.
		\end{equation*}
		In order to show the reverse inequality we proceed as follows:
		Suppose that the $ d $-dimensional random vector $ \mathbf{X} $ has distribution $ \vartheta $ and suppose that $ \mathbf{X}^1, \mathbf{X}^2, \ldots $ is an independent sample from $ \mathbf{X} $.
		Then according to the (multivariate) Glivenko Cantelli theorem (see, e.g., \cite{Va-We}) the sequence $ (\vartheta_{n})_{n \in \NN} $ of empirical measures
		\begin{equation*}
			\vartheta_{n} \coloneqq \frac{1}{n} \sum_{i=1}^{n} \delta_{X_i^1,X_i^2,\ldots,X_i^d} \in \mathcal{P}([0,1]^d)
		\end{equation*}
		converges weakly to $ \vartheta $ with probability one.
		Hence, considering
		\begin{equation*}
			\int_{[0,1]^d} f \, \dd\vartheta_n = \frac{1}{n} \sum_{i=1}^{n} f(X_i^1,X_i^2,\ldots,X_i^d)
		\end{equation*}
		directly yields the desired result.
	\end{proof}
	
	Using the fact that for continuous $ f $ the mapping $ \vartheta \mapsto \int_{[0,1]^d} f \, \dd\vartheta $ is continuous on $ \mathcal{P}_\textrm{ds}([0,1]^d) $ with respect to weak convergence, for determining the supremum in $ C_f^d $ it suffices to calculate the supremum over dense subfamilies of $ \mathcal{P}_\textrm{ds}([0,1]^d) $.
	One of the most handy dense subsets of $ \mathcal{P}_\textrm{ds}([0,1]^d) $ is that of equidistant multivariate shuffles of $ M $ (see \cite{Du-Se} and the references therein).
	We will first recall the concept of equidistant shuffles of $ M $ and then formally prove denseness 
	of this class.
	Doing so, we will write $ I_i^N \coloneqq [\frac{i-1}{N},\frac{i}{N}) $ for every $ N \in \NN $ and every $ i \in [N] $.
	
	\begin{definition}
		\label{def:shuffle}
		An element $ \vartheta \in \mathcal{P}_\textrm{ds}([0,1]^d) $ is called an equidistant shuffle of $ M $ with resolution $ N \in \NN $ if and only if there exist $ d-1 $ permutations $ \sigma_2, \ldots, \sigma_d $ of $ [N] $ such that for the $ d-1 $ transformations $ h_2, \ldots, h_d \colon [0,1] \to [0,1] $, defined by
		\begin{equation*}
			h_i(x) \coloneqq \sum_{j=1}^{N} \left( x - \frac{j-1}{N} + \frac{\sigma_i(j)-1}{N} \right) \, \ind{I_j^N}(x),
		\end{equation*}
		the following property holds:
		$ \vartheta $ is the distribution of the random vector 
		\begin{equation}
			\label{eq:shuffle}
			(X_1, h_2(X_1), h_3(X_1), \ldots, h_d(X_1))
		\end{equation}
		where $ X_1 $ is a uniformly $ [0,1] $-distributed random variable.  
	\end{definition}
	
	Notice that each of the transformations $ h_i $ preserves the Lebesgue measure $ \lambda $, i.e., the push-forward $ \lambda^{h_i} $ of $ \lambda $ via $ h_i $ coincides with $ \lambda $.
	It is therefore straightforward to verify that the distribution of the vector 
	\eqref{eq:shuffle} is a $ d $-stochastic measure. 
	Moreover, it might seem natural to also transform the variable $ X_1 $ in equation \eqref{eq:shuffle} via a $ \lambda $-preserving transformation $ h_1 $ -- since $ h_1 \circ X_1 $, however, is uniformly distributed on $ [0,1] $ one would still end up with a shuffle of $ M $ according to Definition~\ref{def:shuffle}.
	For generalized shuffles of copulas (via general $ \lambda $-preserving transformations) we 
	refer to \cite{Du-Se,FS-Tr} and the references therein.
	
	In the sequel we will let $ \mathcal{S}_d^N $ denote the family of all shuffles of $ M $ with resolution $ N $ and refer to
	\begin{equation*}
		\mathcal{S}_d^{\infty} \coloneqq \bigcup_{N=1}^{\infty} \mathcal{S}_d^N \subseteq \mathcal{P}_\textrm{ds}([0,1]^d)
	\end{equation*}
	as the family of all equidistant shuffles of $ M $.
	
	\begin{remark}
		\label{rem:shuffle-subclass}
		Notice that for every $ N_0 \in \NN $ we have
		\begin{equation*}
			\mathcal{S}_d^{\infty} = \bigcup_{N=N_0}^{\infty} \mathcal{S}_d^N
		\end{equation*}
		since obviously every element of $ \mathcal{S}_d^N $ also is an element of $ \mathcal{S}_d^{kN} $ for every $ k \geq 2 $.
		We will use this fact in the next section and only work with the case $ N \geq d $ or even $ N = kd $ for some $ k \in \NN $.
	\end{remark}
	
	\begin{lemma}
		\label{lem:dense}
		$ \mathcal{S}_d^{\infty} $ is dense in the compact metric space $ (\mathcal{P}_\textrm{ds}([0,1]^d),\rho_W) $.
	\end{lemma}
	\begin{proof}
		It has already been mentioned that weak convergence in $ \mathcal{P}_\textrm{ds}([0,1]^d) $ is equivalent to uniform convergence of the corresponding copulas.
		Let $ \vartheta \in \mathcal{P}_\textrm{ds}([0,1]^d) $ be arbitrary but fixed and let $ C_\vartheta $ denote the corresponding copula.
		Then, according to \cite{Ja-Sw-Ve} (also see \cite[Lemma 6.1]{Gr-Ju-Tr}), the so-called empirical (multilinear) copula $ \hat{C}_n $ (i.e., the copula induced by a sample from $ C_\vartheta $ via pseudo-ranks and multilinear interpolation) converges uniformly to $ C_\vartheta $ with probability one.
		For every $ n $, however, it is straightforward to verify that there exists a copula $ A_n $ corresponding to a $ d $-stochastic measure $ \nu_n \in \mathcal{S}_d^n $ fulfilling $ d_\infty(A_n,\hat{C}_n) \leq \frac{1}{n^{d-1}} $ (in fact, one may simply consider the $ M $-interpolation instead of the multilinear one corresponding to $ \Pi $).
		This directly yields that the sequence $ (\nu_n)_{n \in \NN} $ in $ \mathcal{S}_d^{\infty} $ converges weakly to $ \vartheta $ and the proof is complete.     
	\end{proof}

	\section{Calculating the upper bounds $ C_f^d $ for the $ d $-stochastic setting}
	\label{sec:bounds-d-stoch}
	
	In this section we will derive our main results regarding the $ d $-stochastic setting and 
	determine
	\begin{equation*}
		C_f^d = \sup_{\vartheta \in \mathcal{P}_\textrm{ds}([0,1]^d)} \int_{[0,1]^d} f \, \dd\vartheta
	\end{equation*}
	for each of the three functions $ f_{\min}, f_{\max}, f_{\sum} $.
	We start with the first two functions and work with equidistant shuffles.
	
	\begin{theorem}
		\label{thm:cop-fmin-fmax}
		For the two functions $ f_{\min} $ and $ f_{\max} $, defined according to 
		equations \eqref{eq:f-min} and \eqref{eq:f-max}, respectively, and for every $ d \geq 2 $ 
		we have the identities
		\begin{equation*}
			C_{f_{\min}}^d = \frac{1}{d} \qquad \text{and} \qquad C_{f_{\max}}^d = \frac{d-1}{d}.
		\end{equation*}
	\end{theorem}
	\begin{proof}
		Before going into details we sketch the structure of the proof.
		(i)
		Our first step consists in proving that the stated maximal value is an upper bound for the integral.
		We achieve this goal by using denseness of the family of shuffles, which, in turn, transforms the maximization problem to the realm of permutations.
		We will proceed with step by step maximization and illustrate the arguments/steps for a small example graphically.
		(ii)
		The second step is to provide an example attaining the bound.
		
		\textbf{(i)} By Lemma~\ref{lem:dense}, it suffices to prove the upper bounds for 
		every $ \vartheta \in \mathcal{S}_d^{\infty} $.
		Therefore, let $ d < N \in \NN $ as well as $ \vartheta \in \mathcal{S}_d^N $ be arbitrary but fixed and denote the (piecewise linear) $ \lambda $-preserving transformations according to Definition~\ref{def:shuffle} by $ h_2, \ldots, h_d $.
		Using disintegration we have
		\begin{align*}
			\int_{[0,1]^d} f(\xbf) \, \dd\vartheta(\xbf) &= \int_{[0,1]} \int_{[0,1]^{d-1}} f(\xbf) \, K_{\vartheta}(x_1, \dd(x_2,\ldots,x_d)) \, \dd\lambda(x_1) \\
			&= \int_{[0,1]} f(x_1, h_2(x_1), \ldots, h_d(x_1)) \, \dd\lambda(x_1).
		\end{align*}
		The latter expression can be further simplified for $ f_{\min} $ and $ f_{\max} $:
		Each transformation $ h_i $ with $ i \geq 2 $ corresponds to a unique permutation $ \sigma_i $ of the set $ [N] $ and vice versa.
		In order to simplify notation we will set $ h_1 = \mathrm{id}_{[0,1]} $ as well as $ \sigma_1 = \mathrm{id}_{[N]} $, and set $ \symb_{\min} \coloneqq \min $ as well as $ \symb_{\max} \coloneqq \max $.
		Returning to the integral, for $ k \in \set{\min,\max} $ this yields
		\begin{align*}
			\int_{[0,1]^d} f_k(\xbf) \, \dd\vartheta(\xbf) &= \int_{[0,1]} f_k(x_1, h_2(x_1), \ldots, h_d(x_1)) \, \dd\lambda(x_1) \\
			&= \int_{[0,1]} \mathop{\symb_k}\limits_{i,j \in [d],\ i \neq j} \abs{h_i(x_1) - h_j(x_1)} \, \dd\lambda(x_1) \\
			&= \frac{1}{N} \sum_{n=1}^{N} \mathop{\symb_k}\limits_{i,j \in [d],\ i \neq j} \frac{\abs{\sigma_i(n) - \sigma_j(n)}}{N} \\
			&= \frac{1}{N^2} \sum_{n=1}^{N} \mathop{\symb_k}\limits_{i,j \in [d],\ i \neq j} \abs{\sigma_i(n) - \sigma_j(n)}.
		\end{align*}
		
		\textbf{(i-1)}
		We now consider $ f_{\max} $ in detail and derive an upper bound for 
		\begin{equation}
			\label{eq:sum-of-interest-max}
			\sum_{n=1}^{N} \max\limits_{i,j \in [d],\ i \neq j} \abs{\sigma_i(n) - \sigma_j(n)}.
		\end{equation}
		Making it easier to follow the subsequent arguments, one may view permutations of $ [N] $ as checkerboards with exactly one black field in each row and each column, respectively (see Figure~\ref{fig:permutation}).
		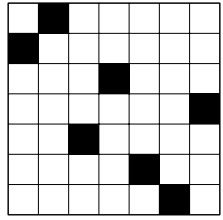
\begin{figure}[h!]
			\centering
			\scalebox{0.4}{
				\begin{tikzpicture}
					\draw[] (0,0) rectangle (7,7);
					\draw[] (0,0) grid (7,7);
					\fill[] (0,5) rectangle ++(1,1);
					\fill[] (1,6) rectangle ++(1,1);
					\fill[] (2,2) rectangle ++(1,1);
					\fill[] (3,4) rectangle ++(1,1);
					\fill[] (4,1) rectangle ++(1,1);
					\fill[] (5,0) rectangle ++(1,1);
					\fill[] (6,3) rectangle ++(1,1);
				\end{tikzpicture}
			}
			\caption{Example of a (checkerboard representation of a) permutation of $ [N] $ with $ N=7 $.}
			\label{fig:permutation}
		\end{figure}
		
		Equation \eqref{eq:sum-of-interest-max} involves exactly $ d $ permutations.
		When overlaying them in one picture, each row can contain at most $ d $ different black squares (note that they could overlap and therefore we have `at most' instead of `exactly').
		We are interested in maximizing the sum over the distances of the highest and lowest black block of each column.
		For addressing this, it is sufficient to depict only the highest and lowest (black) block within each column.
		We will illustrate the subsequent arguments graphically for $ d = 3 $ and $ N = 7 $.
		\begin{figure}[h!]
			\centering
			\scalebox{0.4}{
				\begin{tikzpicture}
					\node[] at (-13.5,1.8) {\scalebox{0.5}{
							\begin{tikzpicture}
								\draw[] (0,0) rectangle (7,7);
								\draw[] (0,0) grid (7,7);
								\fill[] (0,5) rectangle ++(1,1);
								\fill[] (1,6) rectangle ++(1,1);
								\fill[] (2,2) rectangle ++(1,1);
								\fill[] (3,4) rectangle ++(1,1);
								\fill[] (4,1) rectangle ++(1,1);
								\fill[] (5,0) rectangle ++(1,1);
								\fill[] (6,3) rectangle ++(1,1);
							\end{tikzpicture}
					}};
					\node[] at (-9.5,3.5) {\scalebox{0.5}{
							\begin{tikzpicture}
								\draw[] (0,0) rectangle (7,7);
								\draw[] (0,0) grid (7,7);
								\fill[] (0,1) rectangle ++(1,1);
								\fill[] (1,0) rectangle ++(1,1);
								\fill[] (2,5) rectangle ++(1,1);
								\fill[] (3,3) rectangle ++(1,1);
								\fill[] (4,6) rectangle ++(1,1);
								\fill[] (5,4) rectangle ++(1,1);
								\fill[] (6,2) rectangle ++(1,1);
							\end{tikzpicture}
					}};
					\node[] at (-5.5,5.2) {\scalebox{0.5}{
							\begin{tikzpicture}
								\draw[] (0,0) rectangle (7,7);
								\draw[] (0,0) grid (7,7);
								\fill[] (0,1) rectangle ++(1,1);
								\fill[] (1,6) rectangle ++(1,1);
								\fill[] (2,2) rectangle ++(1,1);
								\fill[] (3,3) rectangle ++(1,1);
								\fill[] (4,0) rectangle ++(1,1);
								\fill[] (5,5) rectangle ++(1,1);
								\fill[] (6,4) rectangle ++(1,1);
							\end{tikzpicture}
					}};
					\draw[-stealth, line width=5pt] (-2.5,3.5) -- ++(1.5,0);
					\draw[] (0,0) rectangle (7,7);
					\draw[] (0,0) grid (7,7);
					\fill[] (0,5) rectangle ++(1,1);
					\fill[] (0,1) rectangle ++(1,1);
					\fill[] (1,6) rectangle ++(1,1);
					\fill[] (1,0) rectangle ++(1,1);
					\fill[] (2,5) rectangle ++(1,1);
					\fill[] (2,2) rectangle ++(1,1);
					\fill[] (3,4) rectangle ++(1,1);
					\fill[] (3,3) rectangle ++(1,1);
					\fill[] (4,6) rectangle ++(1,1);
					\fill[] (4,0) rectangle ++(1,1);
					\fill[] (5,5) rectangle ++(1,1);
					\fill[] (5,0) rectangle ++(1,1);
					\fill[] (6,4) rectangle ++(1,1);
					\fill[] (6,2) rectangle ++(1,1);
				\end{tikzpicture}
			}
			\caption{Exemplary aggregation of $ d $ permutations of $ [N] $ for the case $ d=3, N=7 $.}
			\label{fig:agg-three-permutations}
		\end{figure}
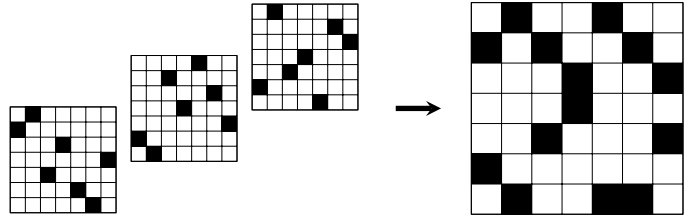
		
		We start with our $ d $ arbitrary permutations (see Figure~\ref{fig:agg-three-permutations}) and first look at the highest row.
		If there are less than $ d $ different black squares in this row, then we move the highest black block of some column(s) not having a black square in the highest row up to the highest row until there are exactly $ d $ different black squares in the highest row or until there is no highest black block in a lower row remaining which could be moved up (see Figure~\ref{fig:opt-high-row}).
		Note that there cannot be more than $ d $ different black blocks in the highest row and thus we would never move highest blocks down instead of up.
		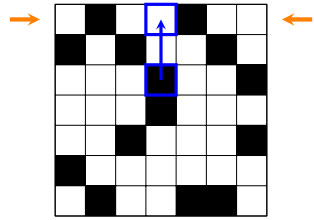
\begin{figure}[h!]
			\centering
			\scalebox{0.4}{
				\begin{tikzpicture}
					\draw[] (0,0) rectangle (7,7);
					\draw[] (0,0) grid (7,7);
					\fill[] (0,5) rectangle ++(1,1);
					\fill[] (0,1) rectangle ++(1,1);
					\fill[] (1,6) rectangle ++(1,1);
					\fill[] (1,0) rectangle ++(1,1);
					\fill[] (2,5) rectangle ++(1,1);
					\fill[] (2,2) rectangle ++(1,1);
					\fill[] (3,4) rectangle ++(1,1);
					\fill[] (3,3) rectangle ++(1,1);
					\fill[] (4,6) rectangle ++(1,1);
					\fill[] (4,0) rectangle ++(1,1);
					\fill[] (5,5) rectangle ++(1,1);
					\fill[] (5,0) rectangle ++(1,1);
					\fill[] (6,4) rectangle ++(1,1);
					\fill[] (6,2) rectangle ++(1,1);
					\draw[-stealth, line width=4pt, orange] (-1.5,6.5) -- ++(1,0);
					\draw[-stealth, line width=4pt, orange] (8.5,6.5) -- ++(-1,0);
					\draw[line width=3pt, blue] (3,4) rectangle ++(1,1);
					\draw[line width=3pt, blue] (3,6) rectangle ++(1,1);
					\draw[-stealth, line width=3pt, blue] (3.5,4.5) -- ++(0,2);
				\end{tikzpicture}
			}
			\caption{Establishing exactly $ d $ squares in the highest row.}
			\label{fig:opt-high-row}
		\end{figure}
		
		This procedure of moving black blocks upwards only increases the sum of distances between the highest and lowest black block in each column and therefore only increases the sum of interest \eqref{eq:sum-of-interest-max}.
		Next we take a look at the second to highest row.
		If there are less than $ d $ different black columns in this row, then we move the highest black block of some column(s), where it currently is located in a lower row, up to the currently handled row until there are exactly $ d $ different black columns in the currently handled row or until there is no highest black block in a lower row remaining which could be moved up (see Figure~\ref{fig:opt-sec-high-row}).
		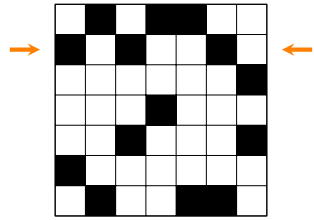
\begin{figure}[h!]
			\centering
			\scalebox{0.4}{
				\begin{tikzpicture}
					\draw[] (0,0) rectangle (7,7);
					\draw[] (0,0) grid (7,7);
					\fill[] (0,5) rectangle ++(1,1);
					\fill[] (0,1) rectangle ++(1,1);
					\fill[] (1,6) rectangle ++(1,1);
					\fill[] (1,0) rectangle ++(1,1);
					\fill[] (2,5) rectangle ++(1,1);
					\fill[] (2,2) rectangle ++(1,1);
					\fill[] (3,6) rectangle ++(1,1);
					\fill[] (3,3) rectangle ++(1,1);
					\fill[] (4,6) rectangle ++(1,1);
					\fill[] (4,0) rectangle ++(1,1);
					\fill[] (5,5) rectangle ++(1,1);
					\fill[] (5,0) rectangle ++(1,1);
					\fill[] (6,4) rectangle ++(1,1);
					\fill[] (6,2) rectangle ++(1,1);
					\draw[-stealth, line width=4pt, orange] (-1.5,5.5) -- ++(1,0);
					\draw[-stealth, line width=4pt, orange] (8.5,5.5) -- ++(-1,0);
				\end{tikzpicture}
			}
			\caption{Exactly $ d=3 $ black squares in the second row.}
			\label{fig:opt-sec-high-row}
		\end{figure}
		
		Again this procedure of moving black blocks only increases the sum of distances between the highest and lowest black block in each column.
		We proceed in this row-wise manner until there are no highest black blocks in a row below the currently handled one left (see Figure~\ref{fig:finish-highest-blocks}).
		\begin{figure}[h!]
			\centering
			\scalebox{0.4}{
				\begin{tikzpicture}
					\draw[] (0,0) rectangle (7,7);
					\draw[] (0,0) grid (7,7);
					\fill[] (0,5) rectangle ++(1,1);
					\fill[] (0,1) rectangle ++(1,1);
					\fill[] (1,6) rectangle ++(1,1);
					\fill[] (1,0) rectangle ++(1,1);
					\fill[] (2,5) rectangle ++(1,1);
					\fill[] (2,2) rectangle ++(1,1);
					\fill[] (3,6) rectangle ++(1,1);
					\fill[] (3,3) rectangle ++(1,1);
					\fill[] (4,6) rectangle ++(1,1);
					\fill[] (4,0) rectangle ++(1,1);
					\fill[] (5,5) rectangle ++(1,1);
					\fill[] (5,0) rectangle ++(1,1);
					\fill[] (6,4) rectangle ++(1,1);
					\fill[] (6,2) rectangle ++(1,1);
					\draw[-stealth, line width=4pt, orange] (-1.5,4.5) -- ++(1,0);
					\draw[-stealth, line width=4pt, orange] (8.5,4.5) -- ++(-1,0);
					\draw[line width=3pt, green!50!black] (-1,4) -- (8,4);
				\end{tikzpicture}
			}
			\caption{Finishing the rearrangment of the highest black blocks.}
			\label{fig:finish-highest-blocks}
		\end{figure}
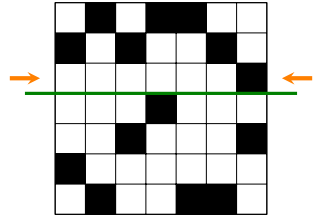
		
		At this point we stop looking at the highest block in the columns and switch to looking at the lowest black block within each column.
		Here we proceed analogously to the previous steps but move the blocks down instead of up.
		Starting with the lowest row and jumping up line by line, we check for the current row, 
		whether there are less than $ d $ different black blocks.
		In that case we move the lowest block of some columns where it currently is located in a higher row down to the currently handled row until there are exactly $ d $ different black columns in the currently handled row or until there is no lowest black block in a higher row remaining which could be moved down.
		The procedure ends when there are no lowest black blocks in a row above the currently handled one left.
		As before, this procedure of moving black blocks only increases the sum of distances between the highest and lowest black block in each column (see Figure~\ref{fig:opt-lowest-blocks}).
		\begin{figure}[h!]
			\centering
			\scalebox{0.4}{
				\begin{tikzpicture}
					\draw[] (0,0) rectangle (7,7);
					\draw[] (0,0) grid (7,7);
					\fill[] (0,5) rectangle ++(1,1);
					\fill[] (0,1) rectangle ++(1,1);
					\fill[] (1,6) rectangle ++(1,1);
					\fill[] (1,0) rectangle ++(1,1);
					\fill[] (2,5) rectangle ++(1,1);
					\fill[] (2,2) rectangle ++(1,1);
					\fill[] (3,6) rectangle ++(1,1);
					\fill[] (3,3) rectangle ++(1,1);
					\fill[] (4,6) rectangle ++(1,1);
					\fill[] (4,0) rectangle ++(1,1);
					\fill[] (5,5) rectangle ++(1,1);
					\fill[] (5,0) rectangle ++(1,1);
					\fill[] (6,4) rectangle ++(1,1);
					\fill[] (6,2) rectangle ++(1,1);
					\draw[line width=3pt, blue] (2,2) rectangle ++(1,1);
					\draw[line width=3pt, blue] (2,1) rectangle ++(1,1);
					\draw[-stealth, line width=3pt, blue] (2.5,2.5) -- ++(0,-1);
					\draw[line width=3pt, blue] (6,2) rectangle ++(1,1);
					\draw[line width=3pt, blue] (6,1) rectangle ++(1,1);
					\draw[-stealth, line width=3pt, blue] (6.5,2.5) -- ++(0,-1);
					\draw[line width=3pt, blue] (3,3) rectangle ++(1,1);
					\draw[line width=3pt, blue] (3,2) rectangle ++(1,1);
					\draw[-stealth, line width=3pt, blue] (3.5,3.5) -- ++(0,-1);
				\end{tikzpicture}
			}
			\caption{Optimize the lowest black blocks.}
			\label{fig:opt-lowest-blocks}
		\end{figure}
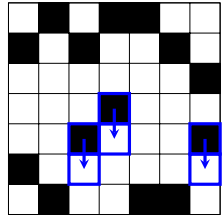
		
		This procedure yields the situation in which the row containing the lowest one of the highest black blocks of each column is \emph{not below} the row containing the highest one of the lowest black blocks of each column;
		depending on $ d $ and $ N $, it is either above or they are equal.
		For seeing this, note that the highest (or lowest, respectively) blocks of the columns, after the before described tuning, will range over $ \lceil \frac{N}{d} \rceil $ rows (see Figure~\ref{fig:sep-high-low}).
		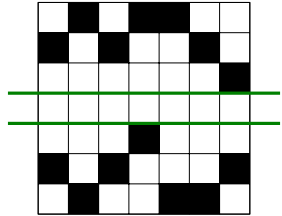
\begin{figure}[h!]
			\centering
			\scalebox{0.4}{
				\begin{tikzpicture}
					\draw[] (0,0) rectangle (7,7);
					\draw[] (0,0) grid (7,7);
					\fill[] (0,5) rectangle ++(1,1);
					\fill[] (0,1) rectangle ++(1,1);
					\fill[] (1,6) rectangle ++(1,1);
					\fill[] (1,0) rectangle ++(1,1);
					\fill[] (2,5) rectangle ++(1,1);
					\fill[] (2,1) rectangle ++(1,1);
					\fill[] (3,6) rectangle ++(1,1);
					\fill[] (3,2) rectangle ++(1,1);
					\fill[] (4,6) rectangle ++(1,1);
					\fill[] (4,0) rectangle ++(1,1);
					\fill[] (5,5) rectangle ++(1,1);
					\fill[] (5,0) rectangle ++(1,1);
					\fill[] (6,4) rectangle ++(1,1);
					\fill[] (6,1) rectangle ++(1,1);
					\draw[line width=3pt, green!50!black] (-1,4) -- (8,4);
					\draw[line width=3pt, green!50!black] (-1,3) -- (8,3);
				\end{tikzpicture}
			}
			\caption{Separation of highest and lowest black blocks.}
			\label{fig:sep-high-low}
		\end{figure}
		
		Starting with a constellation of the afore-mentioned type we now calculate the sum of distances between the highest and lowest black block in each column, which yields an upper bound for the sum in expression \eqref{eq:sum-of-interest-max}.
		If in this configuration we exchange the highest black blocks from two columns, then the sum of distances between the highest and lowest black block in each column will stay constant.
		Namely, the distance in one column increases by exactly the same ammount as the distance in the other column decreases.
		Recall that we only consider the highest and lowest black block of each column and that they are somewhat separated.
		Moreover, this exchanging property analogously holds for the lowest black blocks.
		Together they justify to sort the blocks in such a way that the highest black blocks are in decreasing, and the lowest black blocks are in increasing order (see Figure~\ref{fig:reorder-high-low}).
		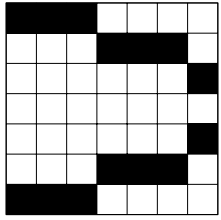
\begin{figure}[h!]
			\centering
			\scalebox{0.4}{
				\begin{tikzpicture}
					\draw[] (0,0) rectangle (7,7);
					\draw[] (0,0) grid (7,7);
					\fill[] (0,6) rectangle ++(1,1);
					\fill[] (0,0) rectangle ++(1,1);
					\fill[] (1,6) rectangle ++(1,1);
					\fill[] (1,0) rectangle ++(1,1);
					\fill[] (2,6) rectangle ++(1,1);
					\fill[] (2,0) rectangle ++(1,1);
					\fill[] (3,5) rectangle ++(1,1);
					\fill[] (3,1) rectangle ++(1,1);
					\fill[] (4,5) rectangle ++(1,1);
					\fill[] (4,1) rectangle ++(1,1);
					\fill[] (5,5) rectangle ++(1,1);
					\fill[] (5,1) rectangle ++(1,1);
					\fill[] (6,4) rectangle ++(1,1);
					\fill[] (6,2) rectangle ++(1,1);
				\end{tikzpicture}
			}
			\caption{Reordering of the highest and lowest black blocks.}
			\label{fig:reorder-high-low}
		\end{figure}
		
		In this situation, however, it is now straightforward to compute the sum $ S $ of distances, which yields an upper bound as desired.
		Writing $ N = dK + R $ with $ 0 \leq R < d $, we get
		\begin{align*}
			\sum_{n=1}^{N} \max\limits_{i,j \in [d],\ i \neq j} \abs{\sigma_i(n) - \sigma_j(n)} &\leq S \\
			&= d \sum_{\ell = 1}^{K} (N - (2\ell-1)) + R (N - (2K+1)) \\
			&= dKN - d \sum_{\ell = 1}^{K} (2\ell-1) + RN - R (2K + 1) \\
			&= N^2 - d K^2 - KR - KR - R \\
			&= N^2 - KN - KR - R \\
			&= N ((d-1)K + R) - KR - R \\
			&= N \cdot \frac{d-1}{d} N + N \cdot \frac{1}{d} R - KR - R \\
			&= \frac{d-1}{d} N^2 - R \left( 1 - \frac{R}{d} \right) \\
			&\leq \frac{d-1}{d} N^2.
		\end{align*}
		Returning to the original integral, this altogether yields
		\begin{equation}
			\label{eq:int-bound-max}
			\int_{[0,1]^d} f_{\max}(\xbf) \, \dd\vartheta(\xbf) = \frac{1}{N^2} \sum_{n=1}^{N} \max\limits_{i,j \in [d],\ i \neq j} \abs{\sigma_i(n) - \sigma_j(n)} \leq \frac{d-1}{d}
		\end{equation}
		which completes the proof for the upper bound concerning $ f_{\max} $.
		
		\textbf{(i-2)}
		Having established this upper bound helps us to determine an upper bound for the corresponding integral over the function $ f_{\min} $.
		We now aim at bounding
		\begin{equation*}
			\sum_{n=1}^{N} \min\limits_{i,j \in [d],\ i \neq j} \abs{\sigma_i(n) - \sigma_j(n)}
		\end{equation*}
		from above.
		As before, we will argue with shifting `black blocks'.
		As soon as the highest and lowest black block for a column is fixed, the minimum distance of two blocks in this column is at most $ \frac{1}{d-1} $ times the distance of the highest and lowest black block in this column (equidistant placing).
		Even if this cannot always be realized, it still yields a numerical upper bound.
		Summing up over all columns yields
		\begin{equation*}
			\sum_{n=1}^{N} \min\limits_{i,j \in [d],\ i \neq j} \abs{\sigma_i(n) - \sigma_j(n)} \leq \frac{1}{d-1} \sum_{n=1}^{N} \max\limits_{i,j \in [d],\ i \neq j} \abs{\sigma_i(n) - \sigma_j(n)}.
		\end{equation*}
		Therefore, using \eqref{eq:int-bound-max}, we end with
		\begin{equation*}
			\int_{[0,1]^d} f_{\min}(\xbf) \, \dd\vartheta(\xbf) = \frac{1}{N^2} \sum_{n=1}^{N} \min\limits_{i,j \in [d],\ i \neq j} \abs{\sigma_i(n) - \sigma_j(n)} \leq \frac{1}{d}
		\end{equation*}
		which completes the proof of the upper bound for $ f_{\min} $.
		
		\textbf{(ii)}
		The last step consists in providing examples showing that the established bounds are already sharp.
		For this we consider (the rotations)
		\begin{equation}
			\label{eq:equidist-shuffle}
			h_r(t) \coloneqq t + \frac{r-1}{d} - \left\lfloor t + \frac{r-1}{d} \right\rfloor
		\end{equation}
		for $ r \in \set{1,\ldots,d} $.
		Plotting these functions jointly in one coordinate system clarifies the reason for choosing shuffles of this type -- the bundle of functions $ h_r $ describes equidistant shifting/rotating (see Figure~\ref{fig:equidist-shuffle}).
		\begin{figure}[h!]
			\centering
			\scalebox{0.8}{
				\begin{tikzpicture}
					\draw[] (0,0) rectangle (3,3);
					\draw[] (0,0) grid (3,3);
					\draw[line width=3pt, blue] (0,0) -- (3,3);
					\draw[line width=3pt, green] (0,1) -- (2,3);
					\draw[line width=3pt, green] (2,0) -- (3,1);
					\draw[line width=3pt, red] (0,2) -- (1,3);
					\draw[line width=3pt, red] (1,0) -- (3,2);
				\end{tikzpicture}
			}
			\caption{The rotations considered in step (ii) for the special case $ d=3 $.}
			\label{fig:equidist-shuffle}
		\end{figure}
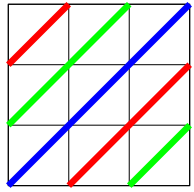
		Letting $ \vartheta^* \in \mathcal{P}_\textrm{ds}([0,1]^d)$ denote the corresponding $ d $-stochastic measure, a straightforward calculation yields
		\begin{equation*}
			\int_{[0,1]^d} f_{\min}(\xbf) \, \dd\vartheta^*(\xbf) = \int_{[0,1]} f_{\min}(x_1, h_2(x_1), \ldots, h_d(x_1)) \, \dd\lambda(x_1) = \frac{1}{d}
		\end{equation*}
		as well as
		\begin{equation*}
			\int_{[0,1]^d} f_{\max}(\xbf) \, \dd\vartheta^*(\xbf) = \int_{[0,1]} f_{\max}(x_1, h_2(x_1), \ldots, h_d(x_1)) \, \dd\lambda(x_1) = \frac{d-1}{d}.
		\end{equation*}
		In other words: the upper bounds are attainable and the proof is complete.
		Figure~\ref{fig:sample-shuffle} depicts a sample from $ \vartheta^* $ as well as its bivariate marginal samples.
	\end{proof}
	\begin{figure}[h!]
		\centering
		\includegraphics[width=0.7\linewidth]{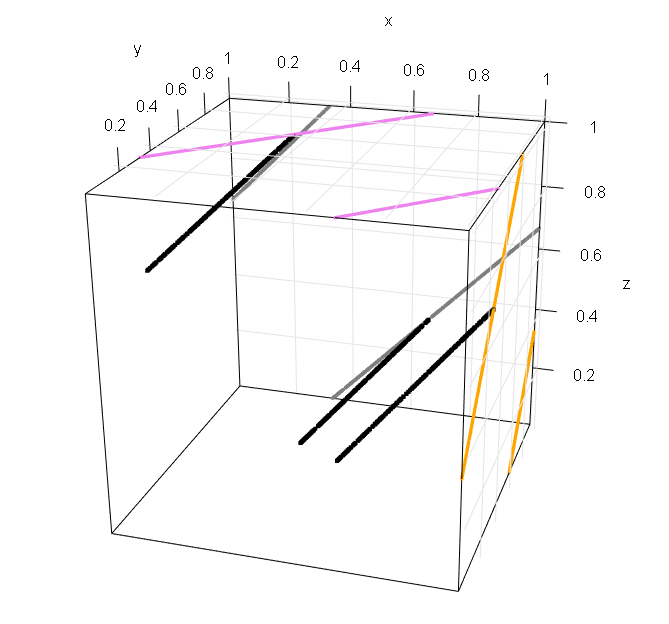}
		\caption{Sample (black) of size $ 1.000 $ of the $ 3 $-stochastic measure $ \vartheta^* $ attaining the upper bound $ \frac{1}{3} $ for $ f_{\min} $ and $ \frac{2}{3} $ for $ f_{\max} $ (as considered in step (ii) in the proof of Theorem~\ref{thm:cop-fmin-fmax}). The gray/orange/pink points denote the bivariate marginal samples.}
		\label{fig:sample-shuffle}
	\end{figure}
	
	Notice that the above proof also shows that the maximal value for the integral can in case of equidistant shuffles only be attained if the resolution $ N $ is a multiple of the dimension $ d $.
	When proving the following theorem concerning the function $ f_{\sum} $, we have to use a slightly different proof strategy due to higher complexity of the involved function.
	
	\begin{theorem}
		\label{thm:cop-fsum}
		For the function $ f_{\sum} $ as defined in \eqref{eq:f-sum} and every $ d \geq 2 $ we have
		\begin{equation*}
			C_{f_{\sum}}^d = \frac{d^2-1}{3}.
		\end{equation*}
	\end{theorem}
	\begin{proof}
		As before we start with sketching the proof structure:
		(i)
		Similar to the proof of Theorem~\ref{thm:cop-fmin-fmax} we will use denseness of shuffles to translate the problem to the language of permutations.
		Here, once more we will step by step maximize the function of interest to establish an upper bound.
		(ii)
		The obtained maximal configuration then turns out to be induced by a well known shuffle, which completes the proof.
		
		\textbf{(i)}
		Using continuity of $ f_{\sum} $, according to Lemma~\ref{lem:dense} it suffices to prove the upper bound for every $ \vartheta \in \mathcal{S}_d^{\infty} $.
		According to Remark~\ref{rem:shuffle-subclass}, without loss of generality we may consider the case $ N = d\ell $ for some integer $ \ell $.
		Let $ \vartheta \in \mathcal{S}_d^N $ be arbitrary but fixed and again let $ h_2, \ldots, h_d $ denote the corresponding $ \lambda $-preserving transformations.
		Analogously to the proof of Theorem~\ref{thm:cop-fmin-fmax}, we then get
		\begin{align*}
			\int_{[0,1]^d} f_{\sum}(\xbf) \, \dd\vartheta(\xbf) &= \int_{[0,1]} f_{\sum}(x_1, h_2(x_1), \ldots, h_d(x_1)) \, \dd\lambda(x_1) \\
			&= \frac{1}{N^2} \sum_{n=1}^{N} \sum\limits_{i,j \in [d]} \abs{\sigma_i(n) - \sigma_j(n)},
		\end{align*}
		where the permutations $ \sigma_i $ as above correspond to the transformations $ h_i $ (to keep notation simple we will again write $ h_1 = \mathrm{id}_{[0,1]} $ as well as $ \sigma_1 = \mathrm{id}_{[N]} $).
		We want to derive an upper bound for
		\begin{equation*}
			\sum_{n=1}^{N} \sum\limits_{i,j \in [d]} \abs{\sigma_i(n) - \sigma_j(n)}
		\end{equation*}
		and proceed as follows again working with `black blocks'.
		For the illustration of our arguments we consider the special case $ d = 4 $ and $ \ell = 2 $.
		In the same way as for the function $ f_{\max} $ inside the proof of Theorem~\ref{thm:cop-fmin-fmax} we first overlay all $ d $ permutations in one joint picture.
		Contrary to before, however, we are now interested in all differences, and therefore keep all the blocks, not just the highest and lowest one.
		It is likely that there are some overlaps.
		To indicate overlaps we write the corresponding numbers into the black blocks (see Figure~\ref{fig:agg-four-permutations}).
		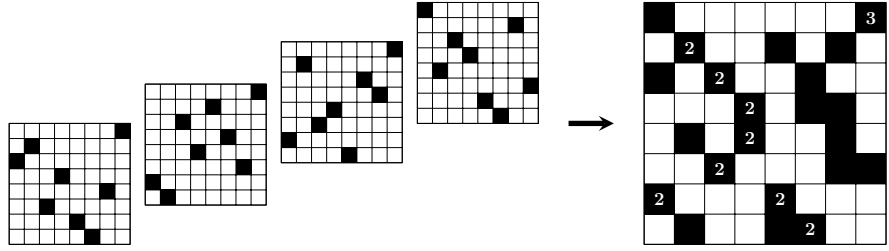
\begin{figure}[h!]
			\centering
			\scalebox{0.4}{
				\begin{tikzpicture}
					\node[] at (-19,2) {\scalebox{0.5}{
							\begin{tikzpicture}
								\draw[] (0,0) rectangle (8,8);
								\draw[] (0,0) grid (8,8);
								\fill[] (0,5) rectangle ++(1,1);
								\fill[] (1,6) rectangle ++(1,1);
								\fill[] (2,2) rectangle ++(1,1);
								\fill[] (3,4) rectangle ++(1,1);
								\fill[] (4,1) rectangle ++(1,1);
								\fill[] (5,0) rectangle ++(1,1);
								\fill[] (6,3) rectangle ++(1,1);
								\fill[] (7,7) rectangle ++(1,1);
							\end{tikzpicture}
					}};
					\node[] at (-14.5,3.3) {\scalebox{0.5}{
							\begin{tikzpicture}
								\draw[] (0,0) rectangle (8,8);
								\draw[] (0,0) grid (8,8);
								\fill[] (0,1) rectangle ++(1,1);
								\fill[] (1,0) rectangle ++(1,1);
								\fill[] (2,5) rectangle ++(1,1);
								\fill[] (3,3) rectangle ++(1,1);
								\fill[] (4,6) rectangle ++(1,1);
								\fill[] (5,4) rectangle ++(1,1);
								\fill[] (6,2) rectangle ++(1,1);
								\fill[] (7,7) rectangle ++(1,1);
							\end{tikzpicture}
					}};
					\node[] at (-10,4.7) {\scalebox{0.5}{
							\begin{tikzpicture}
								\draw[] (0,0) rectangle (8,8);
								\draw[] (0,0) grid (8,8);
								\fill[] (0,1) rectangle ++(1,1);
								\fill[] (1,6) rectangle ++(1,1);
								\fill[] (2,2) rectangle ++(1,1);
								\fill[] (3,3) rectangle ++(1,1);
								\fill[] (4,0) rectangle ++(1,1);
								\fill[] (5,5) rectangle ++(1,1);
								\fill[] (6,4) rectangle ++(1,1);
								\fill[] (7,7) rectangle ++(1,1);
							\end{tikzpicture}
					}};
					\node[] at (-5.5,6) {\scalebox{0.5}{
							\begin{tikzpicture}
								\draw[] (0,0) rectangle (8,8);
								\draw[] (0,0) grid (8,8);
								\fill[] (0,7) rectangle ++(1,1);
								\fill[] (1,3) rectangle ++(1,1);
								\fill[] (2,5) rectangle ++(1,1);
								\fill[] (3,4) rectangle ++(1,1);
								\fill[] (4,1) rectangle ++(1,1);
								\fill[] (5,0) rectangle ++(1,1);
								\fill[] (6,6) rectangle ++(1,1);
								\fill[] (7,2) rectangle ++(1,1);
							\end{tikzpicture}
					}};
					\draw[-stealth, line width=5pt] (-2.5,4) -- ++(1.5,0);
					\draw[] (0,0) rectangle (8,8);
					\draw[] (0,0) grid (8,8);
					\fill[] (0,5) rectangle ++(1,1);
					\fill[] (1,6) rectangle ++(1,1);
					\fill[] (2,2) rectangle ++(1,1);
					\fill[] (3,4) rectangle ++(1,1);
					\fill[] (4,1) rectangle ++(1,1);
					\fill[] (5,0) rectangle ++(1,1);
					\fill[] (6,3) rectangle ++(1,1);
					\fill[] (7,7) rectangle ++(1,1);
					\fill[] (0,1) rectangle ++(1,1);
					\fill[] (1,0) rectangle ++(1,1);
					\fill[] (2,5) rectangle ++(1,1);
					\fill[] (3,3) rectangle ++(1,1);
					\fill[] (4,6) rectangle ++(1,1);
					\fill[] (5,4) rectangle ++(1,1);
					\fill[] (6,2) rectangle ++(1,1);
					\fill[] (7,7) rectangle ++(1,1);
					\fill[] (0,1) rectangle ++(1,1);
					\fill[] (1,6) rectangle ++(1,1);
					\fill[] (2,2) rectangle ++(1,1);
					\fill[] (3,3) rectangle ++(1,1);
					\fill[] (4,0) rectangle ++(1,1);
					\fill[] (5,5) rectangle ++(1,1);
					\fill[] (6,4) rectangle ++(1,1);
					\fill[] (7,7) rectangle ++(1,1);
					\fill[] (0,7) rectangle ++(1,1);
					\fill[] (1,3) rectangle ++(1,1);
					\fill[] (2,5) rectangle ++(1,1);
					\fill[] (3,4) rectangle ++(1,1);
					\fill[] (4,1) rectangle ++(1,1);
					\fill[] (5,0) rectangle ++(1,1);
					\fill[] (6,6) rectangle ++(1,1);
					\fill[] (7,2) rectangle ++(1,1);
					\node[white] at (5.5,0.5) {\huge $ \mathbf{2} $};
					\node[white] at (0.5,1.5) {\huge $ \mathbf{2} $};
					\node[white] at (4.5,1.5) {\huge $ \mathbf{2} $};
					\node[white] at (2.5,2.5) {\huge $ \mathbf{2} $};
					\node[white] at (3.5,3.5) {\huge $ \mathbf{2} $};
					\node[white] at (3.5,4.5) {\huge $ \mathbf{2} $};
					\node[white] at (2.5,5.5) {\huge $ \mathbf{2} $};
					\node[white] at (1.5,6.5) {\huge $ \mathbf{2} $};
					\node[white] at (7.5,7.5) {\huge $ \mathbf{3} $};
				\end{tikzpicture}
			}
			\caption{Aggregation of $ d = 4 $ permutations with some overlaps.}
			\label{fig:agg-four-permutations}
		\end{figure}
		
		We partition the grid into $ d $ groups of $ \ell $ adjacent rows, i.e., the first $ \ell $ rows form a group, the next $ \ell $ rows another one, and so on.
		Next we introduce the notion of the `level' of a black block.
		All blocks in the highest group as well as those in the lowest group have level $ 1 $.
		Blocks located in the second to highest or second to lowest group have level $ 2 $, and so on.
		Rows that already have a (lower) level number are not releveled (see Figure~\ref{fig:bb-levels}).
		\begin{figure}[h!]
			\centering
			\scalebox{0.4}{
				\begin{tikzpicture}
					\draw[] (0,0) rectangle (8,8);
					\draw[] (0,0) grid (8,8);
					\fill[] (0,5) rectangle ++(1,1);
					\fill[] (1,6) rectangle ++(1,1);
					\fill[] (2,2) rectangle ++(1,1);
					\fill[] (3,4) rectangle ++(1,1);
					\fill[] (4,1) rectangle ++(1,1);
					\fill[] (5,0) rectangle ++(1,1);
					\fill[] (6,3) rectangle ++(1,1);
					\fill[] (7,7) rectangle ++(1,1);
					\fill[] (0,1) rectangle ++(1,1);
					\fill[] (1,0) rectangle ++(1,1);
					\fill[] (2,5) rectangle ++(1,1);
					\fill[] (3,3) rectangle ++(1,1);
					\fill[] (4,6) rectangle ++(1,1);
					\fill[] (5,4) rectangle ++(1,1);
					\fill[] (6,2) rectangle ++(1,1);
					\fill[] (7,7) rectangle ++(1,1);
					\fill[] (0,1) rectangle ++(1,1);
					\fill[] (1,6) rectangle ++(1,1);
					\fill[] (2,2) rectangle ++(1,1);
					\fill[] (3,3) rectangle ++(1,1);
					\fill[] (4,0) rectangle ++(1,1);
					\fill[] (5,5) rectangle ++(1,1);
					\fill[] (6,4) rectangle ++(1,1);
					\fill[] (7,7) rectangle ++(1,1);
					\fill[] (0,7) rectangle ++(1,1);
					\fill[] (1,3) rectangle ++(1,1);
					\fill[] (2,5) rectangle ++(1,1);
					\fill[] (3,4) rectangle ++(1,1);
					\fill[] (4,1) rectangle ++(1,1);
					\fill[] (5,0) rectangle ++(1,1);
					\fill[] (6,6) rectangle ++(1,1);
					\fill[] (7,2) rectangle ++(1,1);
					\node[white] at (5.5,0.5) {\huge $ \mathbf{2} $};
					\node[white] at (0.5,1.5) {\huge $ \mathbf{2} $};
					\node[white] at (4.5,1.5) {\huge $ \mathbf{2} $};
					\node[white] at (2.5,2.5) {\huge $ \mathbf{2} $};
					\node[white] at (3.5,3.5) {\huge $ \mathbf{2} $};
					\node[white] at (3.5,4.5) {\huge $ \mathbf{2} $};
					\node[white] at (2.5,5.5) {\huge $ \mathbf{2} $};
					\node[white] at (1.5,6.5) {\huge $ \mathbf{2} $};
					\node[white] at (7.5,7.5) {\huge $ \mathbf{3} $};
					\draw[line width=3pt, green!50!black] (-1,2) -- ++(10,0);
					\draw[line width=3pt, green!50!black] (-1,4) -- ++(10,0);
					\draw[line width=3pt, green!50!black] (-1,6) -- ++(10,0);
					\node[green!50!black,right] at (9,1) {\Huge $ \mathbf{level\ 1} $};
					\node[green!50!black,right] at (9,3) {\Huge $ \mathbf{level\ 2} $};
					\node[green!50!black,right] at (9,5) {\Huge $ \mathbf{level\ 2} $};
					\node[green!50!black,right] at (9,7) {\Huge $ \mathbf{level\ 1} $};
				\end{tikzpicture}
			}
			\caption{Levels of black blocks.}
			\label{fig:bb-levels}
		\end{figure}
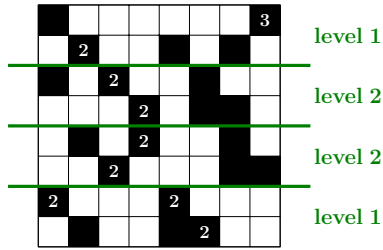
		
		Initially, no group is locked.
		We start with optimizing/shifting the level $ 1 $ black blocks and start with those located in the upper group.
		If there is exactly one of them in each column, then we have nothing to modify und consider this group as locked.
		Otherwise there is an overlapping or there is one column in which black blocks are at least in two different rows.
		In case of an overlapping, we move a black block from a field with multiple blocks to the same row in a column where currently no black block appears in the upper level $ 1 $ group.
		In case of a column with black blocks in different rows of the upper level $ 1 $ group, we move one of them to the same row in a column where currently no black block appears in the upper level $ 1 $ group.
		In order to keep the number of blocks per column constant at the value $ d $, we have to move back (i.e., from the column whose block number was just increased to the column whose block number was just reduced) a black block in another row.
		For doing this take into account that temporarely we created a column with more than $ d $ black blocks.
		Hence there must be another, no already locked group, with a black field in this column; take the highest such black block for moving back (not the just added one).
		This procedure cannot decrease the sum of distances due to the following fact: 
		Inside the column where we first added a block, in fact the highest block was moved up.
		Thus the distances from this highest block to each of the $ d-1 $ other blocks in this column increased by $ \delta $, the row difference between the new and old position of the highest block.
		All other distances in this column remain unchanged.
		Inside the other modified column each of the $ d-1 $ distances involving the modified block decrease at most by $ \delta $ (due to the triangle inequality), whereas the other ones remain unchanged. Thus the sum over all cannot decrease.
		We repeat the previous steps until there is exactly one black block in each column of the upper level $ 1 $ group (see Figure~\ref{fig:opt-upper-level-one}).
		Then we lock this group.
		\begin{figure}[h!]
			\centering
			\scalebox{0.4}{
				\begin{tikzpicture}
					\draw[] (0,0) rectangle (8,8);
					\draw[] (0,0) grid (8,8);
					\fill[] (0,5) rectangle ++(1,1);
					\fill[] (1,6) rectangle ++(1,1);
					\fill[] (2,2) rectangle ++(1,1);
					\fill[] (3,4) rectangle ++(1,1);
					\fill[] (4,1) rectangle ++(1,1);
					\fill[] (5,0) rectangle ++(1,1);
					\fill[] (6,3) rectangle ++(1,1);
					\fill[] (7,7) rectangle ++(1,1);
					\fill[] (0,1) rectangle ++(1,1);
					\fill[] (1,0) rectangle ++(1,1);
					\fill[] (2,5) rectangle ++(1,1);
					\fill[] (3,3) rectangle ++(1,1);
					\fill[] (4,6) rectangle ++(1,1);
					\fill[] (5,4) rectangle ++(1,1);
					\fill[] (6,2) rectangle ++(1,1);
					\fill[] (7,7) rectangle ++(1,1);
					\fill[] (0,1) rectangle ++(1,1);
					\fill[] (1,6) rectangle ++(1,1);
					\fill[] (2,2) rectangle ++(1,1);
					\fill[] (3,3) rectangle ++(1,1);
					\fill[] (4,0) rectangle ++(1,1);
					\fill[] (5,5) rectangle ++(1,1);
					\fill[] (6,4) rectangle ++(1,1);
					\fill[] (7,7) rectangle ++(1,1);
					\fill[] (0,7) rectangle ++(1,1);
					\fill[] (1,3) rectangle ++(1,1);
					\fill[] (2,5) rectangle ++(1,1);
					\fill[] (3,4) rectangle ++(1,1);
					\fill[] (4,1) rectangle ++(1,1);
					\fill[] (5,0) rectangle ++(1,1);
					\fill[] (6,6) rectangle ++(1,1);
					\fill[] (7,2) rectangle ++(1,1);
					\node[white] at (5.5,0.5) {\huge $ \mathbf{2} $};
					\node[white] at (0.5,1.5) {\huge $ \mathbf{2} $};
					\node[white] at (4.5,1.5) {\huge $ \mathbf{2} $};
					\node[white] at (2.5,2.5) {\huge $ \mathbf{2} $};
					\node[white] at (3.5,3.5) {\huge $ \mathbf{2} $};
					\node[white] at (3.5,4.5) {\huge $ \mathbf{2} $};
					\node[white] at (2.5,5.5) {\huge $ \mathbf{2} $};
					\node[white] at (1.5,6.5) {\huge $ \mathbf{2} $};
					\node[white] at (7.5,7.5) {\huge $ \mathbf{3} $};
					\draw[line width=3pt, green!50!black] (-1,2) -- ++(10,0);
					\draw[line width=3pt, green!50!black] (-1,4) -- ++(10,0);
					\draw[line width=3pt, green!50!black] (-1,6) -- ++(10,0);
					\node[green!50!black,right] at (9,1) {\Huge $ \mathbf{level\ 1} $};
					\node[green!50!black,right] at (9,3) {\Huge $ \mathbf{level\ 2} $};
					\node[green!50!black,right] at (9,5) {\Huge $ \mathbf{level\ 2} $};
					\node[green!50!black,right] at (9,7) {\Huge $ \mathbf{level\ 1} $};
					\draw[line width=3pt, blue] (7,7) rectangle ++(1,1);
					\draw[line width=3pt, blue] (3,7) rectangle ++(1,1);
					\draw[-stealth, line width=3pt, blue] (7.5,7.5) -- ++(-4,0);
					\draw[line width=3pt, blue] (3,4) rectangle ++(1,1);
					\draw[line width=3pt, blue] (7,4) rectangle ++(1,1);
					\draw[-stealth, line width=3pt, blue] (3.5,4.5) -- ++(4,0);
					\draw[line width=3pt, blue] (7,7) rectangle ++(1,1);
					\draw[line width=3pt, blue] (5,7) rectangle ++(1,1);
					\draw[-stealth, line width=3pt, blue] (7.5,7.5) -- ++(-2,0);
					\draw[line width=3pt, blue] (5,5) rectangle ++(1,1);
					\draw[line width=3pt, blue] (7,5) rectangle ++(1,1);
					\draw[-stealth, line width=3pt, blue] (5.5,5.5) -- ++(2,0);
					\draw[line width=3pt, blue] (1,6) rectangle ++(1,1);
					\draw[line width=3pt, blue] (2,6) rectangle ++(1,1);
					\draw[-stealth, line width=3pt, blue] (1.5,6.5) -- ++(1,0);
					\draw[line width=3pt, blue] (2,5) rectangle ++(1,1);
					\draw[line width=3pt, blue] (1,5) rectangle ++(1,1);
					\draw[-stealth, line width=3pt, blue] (2.5,5.5) -- ++(-1,0);
				\end{tikzpicture}
			}
			\caption{Optimizing the upper level $ 1 $ group.}
			\label{fig:opt-upper-level-one}
		\end{figure}
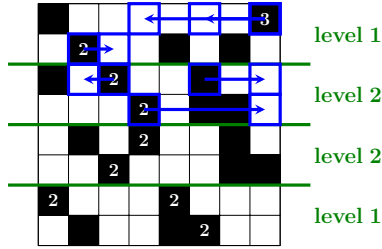
		
		We proceed with considering the level $ 1 $ black blocks in the lower group.
		If there is exactly one of them in each column, then we have nothing to modify and consider this group as locked.
		Otherwise there is again an overlapping or there is one column with black blocks in at least two different rows.
		In case of an overlapping, we move a black block from a field with multiple blocks to the same row in a column where currently no black block appears in the lower level $ 1 $ group.
		In case of a column with black blocks in different rows of the lower level $ 1 $ group we move one of them to the same row in a column where currently no black block appears in the lower level $ 1 $ group.
		In order to keep the number of blocks per column constant at the value $ d $, we again have to move back a black block to another row.
		As before, we temporarily created a column with more than $ d $ black blocks.
		Hence there must be another group, which is not already locked, with a black field in this column; take the lowest such black block for moving back (not the just added one).
		This procedure cannot decrease the sum of distances because of the following behaviour.
		Inside the column where we first added a block, in fact the lowest block was moved down.
		Thus the distances from this lowest block to each of the $ d-1 $ other blocks in this column increased by $ \delta $, the row difference between the new and old position of the lowest block.
		All other distances in this column remain unchanged.
		Inside the other modified column each of the $ d-1 $ distances involving the modified block decreases at most by $ \delta $ (triangle inequality), whereas the other ones remain unchanged as well.
		Thus the sum over all cannot decrease.
		We repeat the previous steps until there is exactly one black block in each column of the lower level $ 1 $ group (see Figure~\ref{fig:opt-lower-level-one}).
		Then we lock this group.
		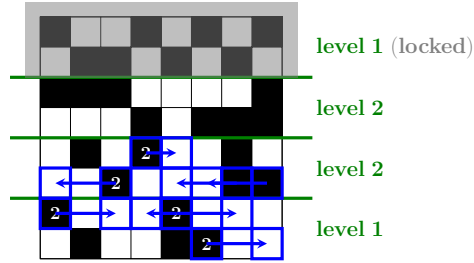
\begin{figure}[h!]
			\centering
			\scalebox{0.4}{
				\begin{tikzpicture}
					\draw[] (0,0) rectangle (8,8);
					\draw[] (0,0) grid (8,8);
					\fill[] (0,5) rectangle ++(1,1);
					\fill[] (1,6) rectangle ++(1,1);
					\fill[] (2,2) rectangle ++(1,1);
					\fill[] (3,4) rectangle ++(1,1);
					\fill[] (4,1) rectangle ++(1,1);
					\fill[] (5,0) rectangle ++(1,1);
					\fill[] (6,3) rectangle ++(1,1);
					\fill[] (7,7) rectangle ++(1,1);
					\fill[] (0,1) rectangle ++(1,1);
					\fill[] (1,0) rectangle ++(1,1);
					\fill[] (2,5) rectangle ++(1,1);
					\fill[] (3,3) rectangle ++(1,1);
					\fill[] (4,6) rectangle ++(1,1);
					\fill[] (5,4) rectangle ++(1,1);
					\fill[] (6,2) rectangle ++(1,1);
					\fill[] (3,7) rectangle ++(1,1);
					\fill[] (0,1) rectangle ++(1,1);
					\fill[] (2,6) rectangle ++(1,1);
					\fill[] (2,2) rectangle ++(1,1);
					\fill[] (3,3) rectangle ++(1,1);
					\fill[] (4,0) rectangle ++(1,1);
					\fill[] (7,5) rectangle ++(1,1);
					\fill[] (6,4) rectangle ++(1,1);
					\fill[] (5,7) rectangle ++(1,1);
					\fill[] (0,7) rectangle ++(1,1);
					\fill[] (1,3) rectangle ++(1,1);
					\fill[] (1,5) rectangle ++(1,1);
					\fill[] (7,4) rectangle ++(1,1);
					\fill[] (4,1) rectangle ++(1,1);
					\fill[] (5,0) rectangle ++(1,1);
					\fill[] (6,6) rectangle ++(1,1);
					\fill[] (7,2) rectangle ++(1,1);
					\node[white] at (5.5,0.5) {\huge $ \mathbf{2} $};
					\node[white] at (0.5,1.5) {\huge $ \mathbf{2} $};
					\node[white] at (4.5,1.5) {\huge $ \mathbf{2} $};
					\node[white] at (2.5,2.5) {\huge $ \mathbf{2} $};
					\node[white] at (3.5,3.5) {\huge $ \mathbf{2} $};
					\draw[line width=3pt, green!50!black] (-1,2) -- ++(10,0);
					\draw[line width=3pt, green!50!black] (-1,4) -- ++(10,0);
					\draw[line width=3pt, green!50!black] (-1,6) -- ++(10,0);
					\node[green!50!black,right] at (9,1) {\Huge $ \mathbf{level\ 1} $};
					\node[green!50!black,right] at (9,3) {\Huge $ \mathbf{level\ 2} $};
					\node[green!50!black,right] at (9,5) {\Huge $ \mathbf{level\ 2} $};
					\node[green!50!black,right] at (9,7) {\Huge $ \mathbf{level\ 1\ \textcolor{gray}{(locked)}} $};
					\fill[gray,opacity=0.5] (-0.5,6) rectangle (8.5,8.5);
					\draw[line width=3pt, blue] (0,1) rectangle ++(1,1);
					\draw[line width=3pt, blue] (2,1) rectangle ++(1,1);
					\draw[-stealth, line width=3pt, blue] (0.5,1.5) -- ++(2,0);
					\draw[line width=3pt, blue] (2,2) rectangle ++(1,1);
					\draw[line width=3pt, blue] (0,2) rectangle ++(1,1);
					\draw[-stealth, line width=3pt, blue] (2.5,2.5) -- ++(-2,0);
					\draw[line width=3pt, blue] (4,1) rectangle ++(1,1);
					\draw[line width=3pt, blue] (3,1) rectangle ++(1,1);
					\draw[-stealth, line width=3pt, blue] (4.5,1.5) -- ++(-1,0);
					\draw[line width=3pt, blue] (3,3) rectangle ++(1,1);
					\draw[line width=3pt, blue] (4,3) rectangle ++(1,1);
					\draw[-stealth, line width=3pt, blue] (3.5,3.5) -- ++(1,0);
					\draw[line width=3pt, blue] (4,1) rectangle ++(1,1);
					\draw[line width=3pt, blue] (6,1) rectangle ++(1,1);
					\draw[-stealth, line width=3pt, blue] (4.5,1.5) -- ++(2,0);
					\draw[line width=3pt, blue] (6,2) rectangle ++(1,1);
					\draw[line width=3pt, blue] (4,2) rectangle ++(1,1);
					\draw[-stealth, line width=3pt, blue] (6.5,2.5) -- ++(-2,0);
					\draw[line width=3pt, blue] (5,0) rectangle ++(1,1);
					\draw[line width=3pt, blue] (7,0) rectangle ++(1,1);
					\draw[-stealth, line width=3pt, blue] (5.5,0.5) -- ++(2,0);
					\draw[line width=3pt, blue] (7,2) rectangle ++(1,1);
					\draw[line width=3pt, blue] (5,2) rectangle ++(1,1);
					\draw[-stealth, line width=3pt, blue] (7.5,2.5) -- ++(-2,0);
				\end{tikzpicture}
			}
			\caption{Optimizing the lower level $ 1 $ group.}
			\label{fig:opt-lower-level-one}
		\end{figure}
		
		Now all level $ 1 $ blocks are locked.
		Note the very important fact that for the current configuration the exact position of level $ 1 $ blocks is not relevant for moving blocks of higher level up and down (not entering the locked area) since the sum of the distance to the upper level $ 1 $ block and the distance to the lower level $ 1 $ block is the same for all possible block positions in between, i.e., for all higher level positions.
		Therefore, all level $ 1 $ blocks can be ignored in the optimizing of higher level blocks.
		We proceed level-wise and perform the same steps as before for a not locked group with the lowest level.
		In fact, it does not matter, whether we start with the upper or lower group of that concrete level (see Figure~\ref{fig:opt-next-level}).
		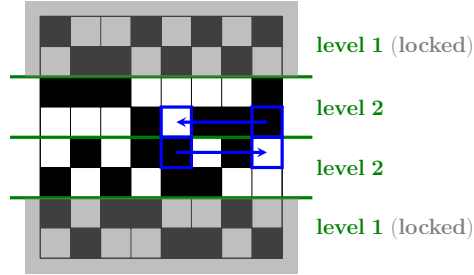
\begin{figure}[h!]
			\centering
			\scalebox{0.4}{
				\begin{tikzpicture}
					\draw[] (0,0) rectangle (8,8);
					\draw[] (0,0) grid (8,8);
					\fill[] (0,5) rectangle ++(1,1);
					\fill[] (1,6) rectangle ++(1,1);
					\fill[] (2,2) rectangle ++(1,1);
					\fill[] (3,4) rectangle ++(1,1);
					\fill[] (6,1) rectangle ++(1,1);
					\fill[] (5,0) rectangle ++(1,1);
					\fill[] (6,3) rectangle ++(1,1);
					\fill[] (7,7) rectangle ++(1,1);
					\fill[] (0,1) rectangle ++(1,1);
					\fill[] (1,0) rectangle ++(1,1);
					\fill[] (2,5) rectangle ++(1,1);
					\fill[] (3,3) rectangle ++(1,1);
					\fill[] (4,6) rectangle ++(1,1);
					\fill[] (5,4) rectangle ++(1,1);
					\fill[] (4,2) rectangle ++(1,1);
					\fill[] (3,7) rectangle ++(1,1);
					\fill[] (2,1) rectangle ++(1,1);
					\fill[] (2,6) rectangle ++(1,1);
					\fill[] (0,2) rectangle ++(1,1);
					\fill[] (4,3) rectangle ++(1,1);
					\fill[] (4,0) rectangle ++(1,1);
					\fill[] (7,5) rectangle ++(1,1);
					\fill[] (6,4) rectangle ++(1,1);
					\fill[] (5,7) rectangle ++(1,1);
					\fill[] (0,7) rectangle ++(1,1);
					\fill[] (1,3) rectangle ++(1,1);
					\fill[] (1,5) rectangle ++(1,1);
					\fill[] (7,4) rectangle ++(1,1);
					\fill[] (3,1) rectangle ++(1,1);
					\fill[] (7,0) rectangle ++(1,1);
					\fill[] (6,6) rectangle ++(1,1);
					\fill[] (5,2) rectangle ++(1,1);
					\draw[line width=3pt, green!50!black] (-1,2) -- ++(10,0);
					\draw[line width=3pt, green!50!black] (-1,4) -- ++(10,0);
					\draw[line width=3pt, green!50!black] (-1,6) -- ++(10,0);
					\node[green!50!black,right] at (9,1) {\Huge $ \mathbf{level\ 1\ \textcolor{gray}{(locked)}} $};
					\node[green!50!black,right] at (9,3) {\Huge $ \mathbf{level\ 2} $};
					\node[green!50!black,right] at (9,5) {\Huge $ \mathbf{level\ 2} $};
					\node[green!50!black,right] at (9,7) {\Huge $ \mathbf{level\ 1\ \textcolor{gray}{(locked)}} $};
					\fill[gray,opacity=0.5] (-0.5,6) rectangle (8.5,8.5);
					\fill[gray,opacity=0.5] (-0.5,-0.5) rectangle (8.5,2);
					\draw[line width=3pt, blue] (7,4) rectangle ++(1,1);
					\draw[line width=3pt, blue] (4,4) rectangle ++(1,1);
					\draw[-stealth, line width=3pt, blue] (7.5,4.5) -- ++(-3,0);
					\draw[line width=3pt, blue] (4,3) rectangle ++(1,1);
					\draw[line width=3pt, blue] (7,3) rectangle ++(1,1);
					\draw[-stealth, line width=3pt, blue] (4.5,3.5) -- ++(3,0);
				\end{tikzpicture}
			}
			\caption{Optimizing the next level.}
			\label{fig:opt-next-level}
		\end{figure}
		
		For even dimension $ d $ the algorithm stops with locking the two groups of highest level.
		For odd dimension $ d $ the algorithm stops with locking the single group of highest level.
		Note that if there is only one unlocked group remaining, then this group automatically satisfies that each column contains exactly one black block.
		In fact, each locked group has exactly one black block per column and at the moment, where a group gets locked, each column contains exactly $ d $ black blocks.
		
		After finitely many steps, all groups are locked.
		We now ignore the previous levels and enumerate the $ d $ groups top-down (see Figure~\ref{fig:opt-groups}).
		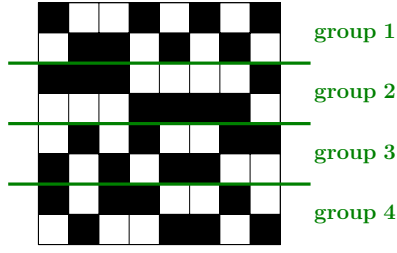
\begin{figure}[h!]
			\centering
			\scalebox{0.4}{
				\begin{tikzpicture}
					\draw[] (0,0) rectangle (8,8);
					\draw[] (0,0) grid (8,8);
					\fill[] (0,5) rectangle ++(1,1);
					\fill[] (1,6) rectangle ++(1,1);
					\fill[] (2,2) rectangle ++(1,1);
					\fill[] (3,4) rectangle ++(1,1);
					\fill[] (6,1) rectangle ++(1,1);
					\fill[] (5,0) rectangle ++(1,1);
					\fill[] (6,3) rectangle ++(1,1);
					\fill[] (7,7) rectangle ++(1,1);
					\fill[] (0,1) rectangle ++(1,1);
					\fill[] (1,0) rectangle ++(1,1);
					\fill[] (2,5) rectangle ++(1,1);
					\fill[] (3,3) rectangle ++(1,1);
					\fill[] (4,6) rectangle ++(1,1);
					\fill[] (5,4) rectangle ++(1,1);
					\fill[] (4,2) rectangle ++(1,1);
					\fill[] (3,7) rectangle ++(1,1);
					\fill[] (2,1) rectangle ++(1,1);
					\fill[] (2,6) rectangle ++(1,1);
					\fill[] (0,2) rectangle ++(1,1);
					\fill[] (7,3) rectangle ++(1,1);
					\fill[] (4,0) rectangle ++(1,1);
					\fill[] (7,5) rectangle ++(1,1);
					\fill[] (6,4) rectangle ++(1,1);
					\fill[] (5,7) rectangle ++(1,1);
					\fill[] (0,7) rectangle ++(1,1);
					\fill[] (1,3) rectangle ++(1,1);
					\fill[] (1,5) rectangle ++(1,1);
					\fill[] (4,4) rectangle ++(1,1);
					\fill[] (3,1) rectangle ++(1,1);
					\fill[] (7,0) rectangle ++(1,1);
					\fill[] (6,6) rectangle ++(1,1);
					\fill[] (5,2) rectangle ++(1,1);
					\draw[line width=3pt, green!50!black] (-1,2) -- ++(10,0);
					\draw[line width=3pt, green!50!black] (-1,4) -- ++(10,0);
					\draw[line width=3pt, green!50!black] (-1,6) -- ++(10,0);
					\node[green!50!black,right] at (9,1) {\Huge $ \mathbf{group\ 4} $};
					\node[green!50!black,right] at (9,3) {\Huge $ \mathbf{group\ 3} $};
					\node[green!50!black,right] at (9,5) {\Huge $ \mathbf{group\ 2} $};
					\node[green!50!black,right] at (9,7) {\Huge $ \mathbf{group\ 1} $};
				\end{tikzpicture}
			}
			\caption{Groups with exactly one black block per column.}
			\label{fig:opt-groups}
		\end{figure}
		
		The previous steps yield a constellation, in which each group contains exactly one black block per column.
		Hence, if inside a group two columns are exchanged, then the sum of distances within one of the columns increases by the same amount as the sum of distances in the other column decreases.
		In other words: reordering the columns group-wisely does not change the total sum of differences.
		This allows us to sort the columns in each group such that they form $ d $ strictly increasing `lines' of length $ \ell $ (see Figure~\ref{fig:reordered-cols}).
		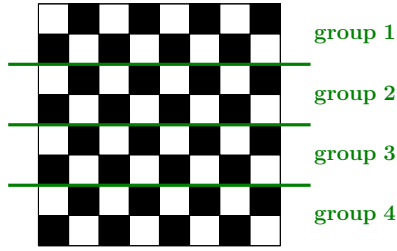
\begin{figure}[h!]
			\centering
			\scalebox{0.4}{
				\begin{tikzpicture}
					\draw[] (0,0) rectangle (8,8);
					\draw[] (0,0) grid (8,8);
					\fill[] (0,0) rectangle ++(1,1);
					\fill[] (2,0) rectangle ++(1,1);
					\fill[] (4,0) rectangle ++(1,1);
					\fill[] (6,0) rectangle ++(1,1);
					\fill[] (1,1) rectangle ++(1,1);
					\fill[] (3,1) rectangle ++(1,1);
					\fill[] (5,1) rectangle ++(1,1);
					\fill[] (7,1) rectangle ++(1,1);
					\fill[] (0,2) rectangle ++(1,1);
					\fill[] (2,2) rectangle ++(1,1);
					\fill[] (4,2) rectangle ++(1,1);
					\fill[] (6,2) rectangle ++(1,1);
					\fill[] (1,3) rectangle ++(1,1);
					\fill[] (3,3) rectangle ++(1,1);
					\fill[] (5,3) rectangle ++(1,1);
					\fill[] (7,3) rectangle ++(1,1);
					\fill[] (0,4) rectangle ++(1,1);
					\fill[] (2,4) rectangle ++(1,1);
					\fill[] (4,4) rectangle ++(1,1);
					\fill[] (6,4) rectangle ++(1,1);
					\fill[] (1,5) rectangle ++(1,1);
					\fill[] (3,5) rectangle ++(1,1);
					\fill[] (5,5) rectangle ++(1,1);
					\fill[] (7,5) rectangle ++(1,1);
					\fill[] (0,6) rectangle ++(1,1);
					\fill[] (2,6) rectangle ++(1,1);
					\fill[] (4,6) rectangle ++(1,1);
					\fill[] (6,6) rectangle ++(1,1);
					\fill[] (1,7) rectangle ++(1,1);
					\fill[] (3,7) rectangle ++(1,1);
					\fill[] (5,7) rectangle ++(1,1);
					\fill[] (7,7) rectangle ++(1,1);
					\draw[line width=3pt, green!50!black] (-1,2) -- ++(10,0);
					\draw[line width=3pt, green!50!black] (-1,4) -- ++(10,0);
					\draw[line width=3pt, green!50!black] (-1,6) -- ++(10,0);
					\node[green!50!black,right] at (9,1) {\Huge $ \mathbf{group\ 4} $};
					\node[green!50!black,right] at (9,3) {\Huge $ \mathbf{group\ 3} $};
					\node[green!50!black,right] at (9,5) {\Huge $ \mathbf{group\ 2} $};
					\node[green!50!black,right] at (9,7) {\Huge $ \mathbf{group\ 1} $};
				\end{tikzpicture}
			}
			\caption{Groups with reordered columns.}
			\label{fig:reordered-cols}
		\end{figure}
		
		\textbf{(ii)}
		The resulting pattern is obviously induced by the shuffles defined in equation \eqref{eq:equidist-shuffle} describing equidistant shifting.
		Altogether we get
		\begin{equation*}
			\max_{\vartheta \in \mathcal{P}_\textrm{ds}([0,1]^d)} \int_{[0,1]^d} f_{\sum} \, \dd\vartheta = \int_{[0,1]} f_{\sum}(x_1, h_2(x_1), \ldots, h_d(x_1)) \, \dd\lambda(x_1)
		\end{equation*}
		and it remains to compute the latter integral.
		A straightforward calculation yields
		\begin{align*}
			\int_{[0,1]} f_{\sum}(x_1, h_2(x_1), \ldots, h_d(x_1)) \, \dd\lambda(x_1) &= f_{\sum}\left( 0, \frac{1}{d}, \frac{2}{d}, \ldots, \frac{d-1}{d} \right) \\
			&= 2 \sum_{j=1}^{d-1} \frac{j}{d} \cdot (d-j) = \frac{2}{d} \sum_{j=1}^{d-1} \left( jd - j^2 \right) \\
			&= \frac{2}{d} \left( d \cdot \frac{(d-1)d}{2} - \frac{(d-1)d(2d-1)}{6} \right) \\
			&= (d-1)d - \frac{(d-1)(2d-1)}{3} = \frac{d^2-1}{3},
		\end{align*}
		and the proof is complete. 
	\end{proof}

	\section{Solving the original maximization problems}
	\label{sec:back-orig-prob}
	
	Having determined the upper bound $ C_f^d $ for each of the three functions $ f_{\min} $, $ f_{\max} $ and $ f_{\sum} $, we now return to the original problem of calculating
	\begin{equation*}
		B_f^d = \sup_{(x_n)_{n \in \NN} \in \mathcal{U}} \limsup\limits_{n \to \infty} \frac{1}{n} \sum_{i=1}^{n} f(x_i,x_{i+1},\ldots,x_{i+d-1}).
	\end{equation*}
	In one of the previous sections we have already proved that $ B_f^d \leq C_f^d $ holds for all
	three functions of interest.
	We now show (the somewhat surprising fact) that $ B_f^d $ and $ C_f^d $ coincide for all three cases.
	For $ f_{\max} $ and $ f_{\sum} $ we will construct a u.f.d.s.\ attaining the upper bounds, while for $ f_{\min} $ we show the existence of u.f.d.s.\ coming arbitrarily close to the upper bound (implying that the upper bound is best possible).
	 
	\begin{proposition}
		\label{prop:fmax-fsum}
		For every $ d \geq 2 $ there exists a u.f.d.s.\ $ (z_n)_{n \in \NN} \in \mathcal{U} $ with 
		\begin{equation*}
			\lim\limits_{n \to \infty} \frac{1}{n} \sum_{i=1}^{n} f_{\max}(z_i,z_{i+1},\ldots,z_{i+d-1}) = \frac{d-1}{d} = C_{f_{\max}}^d
		\end{equation*}
		and
		\begin{equation*}
			\lim\limits_{n \to \infty} \frac{1}{n} \sum_{i=1}^{n} f_{\sum}(z_i,z_{i+1},\ldots,z_{i+d-1}) = \frac{d^2-1}{3} = C_{f_{\sum}}^d.
		\end{equation*}
	\end{proposition}
	\begin{proof}
		Suppose that $ (x_n)_{n \in \NN} \in \mathcal{U} $ is a u.f.d.s..
		Then the sequence $ (y_n)_{n \in \NN} $ defined by $ y_n \coloneqq \frac{x_n}{d} $ obviously is uniformly distributed on the interval $ [0,\frac{1}{d}] $.
		Based on $ (y_n)_{n \in \NN} $ we construct a u.f.d.s.\ $ (z_n)_{n \in \NN} $ as follows.
		Decomposing the index as $ i = kd + r $ with a non-negative integer $ k $ and $ 0 < r \leq d $, set
		\begin{equation}
			\label{eq:zn}
			z_i \coloneqq y_{k+1} + \frac{r-1}{d}.
		\end{equation}
		In other words, $ (z_n)_{n \in \NN} $ has the form
		\begin{align*}
			\left(\vphantom{\frac{d-1}{d}}\right.&y_1,\ y_1 + \frac{1}{d},\ y_1 + \frac{2}{d},\ \ldots,\ y_1 + \frac{d-1}{d},\\
			&y_2,\ y_2 + \frac{1}{d},\ y_2 + \frac{2}{d},\ \ldots,\ y_2 + \frac{d-1}{d},\\
			&y_3,\ \qquad \ldots \qquad \left.\vphantom{\frac{d-1}{d}}\right).
		\end{align*}
		It is straightforward to verify that $ (z_n)_{n \in \NN} $ is a u.f.d.s..
		
		(i)
		We first consider $ f_{\max} $.
		Reusing the index decomposition $ i = kd + r $ from above directly yields
		\begin{equation*}
			f_{\max}(z_i, \ldots, z_{i+d-1}) =
			\begin{cases}
				\frac{d-1}{d} & \text{for } r = 1 \\
				\frac{d-1}{d} + y_{k+1} - y_{k+2} & \text{for } r \geq 2.
			\end{cases}
		\end{equation*}
		Therefore it follows that
		\begin{multline*}
			\frac{1}{n} \sum_{i=1}^{n} f_{\max}(z_i, \ldots, z_{i+d-1}) = \frac{d-1}{d} + \frac{1}{n} \cdot (d-1) \sum_{k=1}^{\left\lfloor \frac{n}{d} \right\rfloor} (y_k - y_{k+1}) \\
			+ \frac{1}{n} \max\setb{0, n - d \left\lfloor \frac{n}{d} \right\rfloor - 1} \cdot \left( y_{\left\lfloor \frac{n}{d} \right\rfloor + 1} - y_{\left\lfloor \frac{n}{d} \right\rfloor + 2} \right)
		\end{multline*}
		which, using of the telescope structure of the sum, yields
		\begin{equation*}
			\abs{\frac{1}{n} \sum_{i=1}^{n} f_{\max}(z_i, \ldots, z_{i+d-1}) - \frac{d-1}{d}} \leq \frac{1}{n} \cdot 2 \cdot (d-1) \cdot \frac{1}{d} \quad \overset{n \to \infty}{\longrightarrow} \quad 0.
		\end{equation*}
		Hence we have found a u.f.d.s.\ $ (z_n)_{n \in \NN} $ satisfying 
		\begin{equation*}
			\lim\limits_{n \to \infty} \frac{1}{n} \sum_{i=1}^{n} f_{\max}(z_i, \ldots, z_{i+d-1}) = \frac{d-1}{d},
		\end{equation*}
		proving the first assertion of our proposition.
		
		(ii)
		We now turn towards $ f_{\sum} $.
		Using the u.f.d.s.\ $ (z_n)_{n \in \NN} $ from the previous step, we get
		\begin{equation*}
			f_{\sum}(z_i, \ldots, z_{i+d-1}) = \frac{d^2-1}{3} + C(d,r) \left( y_{k+1} - y_{k+2} \right)
		\end{equation*}
		where once again the index decomposition $ i = kd + r $ with a non-negative integer $ k $ and $ 0 < r \leq d $ is used; thereby $ C(d,r) $ denotes a constant only depending on the dimension $ d $ and the `sequence column' parameter $ r $.
		In fact, $ C(d,r) $ counts, how often the difference $ y_{k+1} - y_{k+2} $ has to be considered.
		Making use of the appearing telescope structure yields
		\begin{equation*}
			\abs{\frac{1}{n} \sum_{i=1}^{n} f_{\sum}(z_i, \ldots, z_{i+d-1}) - \frac{d^2-1}{3}} \leq \frac{1}{n} \cdot \left( \sum_{j=1}^{d} C(d,j) \right) \cdot 2 \sup_{k \in \NN} y_k \quad \overset{n \to \infty}{\longrightarrow} \quad 0.
		\end{equation*}
		Thus, the u.f.d.s.\ $ (z_n)_{n \in \NN} $ also satisfies
		\begin{equation*}
			\lim\limits_{n \to \infty} \frac{1}{n} \sum_{i=1}^{n} f_{\sum}(z_i, \ldots, z_{i+d-1}) = \frac{d^2-1}{3},
		\end{equation*}
		which concludes the proof.
	\end{proof}
	
	Finally, we revisit the minimum distance function $ f_{\min} $.
	Here the procedure performed for the other two functions seems not work in a comparably elegant way, since we do not have sufficient control over what values the difference $ y_k - y_{k+1} $ can attain (in particular, they are strongly restricted by the uniform distribution property).
	It is more likely that we could stay away from the extrem bound in this case as the proof of the bound has shown that maximal minimum distance requires maximal maximum distance which makes it even more restrictive.
	Furthermore, for $ f_{\max} $ and $ f_{\sum} $ distance values larger and smaller than the average are having an effect (cf.\ the used telescope structure), whereas in the case of $ f_{\min} $ it is likely that only smaller ones will contribute.
	We will prove the following result for $ d \geq 3 $ (which is no restriction, since for $ d = 2 $ all three functions are identical up to the constant $ 2 $).
	
	\begin{proposition}
		\label{prop:fmin}
		For every $ d \geq 3 $ there exists some u.f.d.s.\ $ \left((z_n^L)_{n \in \NN}\right)_{L \in \NN} \subseteq \mathcal{U} $ such that we have 
		\begin{equation*}
			\lim\limits_{L \to \infty} \limsup\limits_{n \to \infty} \frac{1}{n} \sum_{i=1}^{n} f_{\min}(z_i,z_{i+1},\ldots,z_{i+d-1}) = \frac{1}{d} = C_{f_{\min}}^d.
		\end{equation*}
	\end{proposition}
	\begin{proof}
		We reuse the construction of the sequence $ (z_n)_{n \in \NN} $ from the proof of Proposition~\ref{prop:fmax-fsum}, but now do not start with an arbitrary uniformly distributed sequence $ (x_n)_{n \in \NN} \in \mathcal{U} $.
		Allowing for more control, we restrict to sequences $ (x_n)_{n \in \NN} \in \mathcal{U} $ with special properties and construct $ (x_n)_{n \in \NN} $ as follows:
		Assume that a $ (w_n)_{n \in \NN} $ is a u.f.d.s.\ and let $ L\geq 2 $ be an arbitrary, positive integer.
		Then obviously $ (\frac{w_n}{L})_{n \in \NN} $ is uniformly distributed on $ [0,\frac{1}{L}] $.
		As in the construction of $ (z_n)_{n \in \NN} $ before, we split the unit interval into subintervals of the same length.
		This time we use $ L $ intervals of length $ \frac{1}{L} $ each and run through these subintervalls in reverse order, i.e., we define $ (x_n)_{n \in \NN} $ by
		\begin{align*}
			\left(\vphantom{\frac{L-1}{L}}\right.&\frac{w_1}{L} + \frac{L-1}{L},\ \frac{w_1}{L} + \frac{L-2}{L},\ \frac{w_1}{L} + \frac{L-3}{L},\ \ldots,\ \frac{w_1}{L},\\
			&\frac{w_2}{L} + \frac{L-1}{L},\ \frac{w_2}{L} + \frac{L-2}{L},\ \frac{w_2}{L} + \frac{L-3}{L},\ \ldots,\ \frac{w_2}{L},\\
			&\frac{w_3}{L} + \frac{L-1}{L},\ \qquad \ldots \qquad \left.\vphantom{\frac{L-1}{L}}\right).
		\end{align*}
		Based on this sequence $ (x_n)_{n \in \NN} $ now define $ (y_n)_{n \in \NN} $ and $ (z_n)_{n \in \NN} $ as in the proof of Proposition~\ref{prop:fmax-fsum}.
		Again decomposing the index via $ i = kd + r $ with a non-negative integer $ k $ and $ 0 < r \leq d $, we get
		\begin{equation*}
			f_{\min}(z_i, \ldots, z_{i+d-1}) = \frac{1}{d} - s(i)
		\end{equation*}
		where $ s(i) = 0 $ in the cases $ r = 1 $ as well as $ r \geq 2 \land L \nmid k $ and $ s(i) \in [\frac{1}{d} - \frac{2}{dL}, \frac{1}{d}] $ for $ r \geq 2 \land L \mid k $.
		For our bound this yields
		\begin{align*}
			\frac{1}{n} \sum_{i=1}^{n} f_{\min}(z_i, \ldots, z_{i+d-1}) &= \frac{1}{d} - \frac{1}{n} \sum_{i=1}^{n} s(i) \\
			&\geq \frac{1}{d} - \frac{1}{n} \cdot \left\lceil \frac{n}{dL} \right\rceil \cdot (d-1) \cdot \frac{1}{d} \quad \overset{n \to \infty}{\longrightarrow} \quad \frac{1}{d} - \frac{d-1}{d^2 L}
		\end{align*}
		which implies
		\begin{equation*}
			\limsup\limits_{n \to \infty} \frac{1}{n} \sum_{i=1}^{n} f_{\min}(z_i, \ldots, z_{i+d-1}) \geq \frac{1}{d} - \frac{d-1}{d^2 L}.
		\end{equation*}
		Letting $ L $ increase, using the upper bound from Theorem~\ref{thm:cop-fmin-fmax}, 
		we finally get
		\begin{equation*}
			\lim\limits_{L \to \infty} \limsup\limits_{n \to \infty} \frac{1}{n} \sum_{i=1}^{n} f_{\min}(z_i, \ldots, z_{i+d-1}) = \frac{1}{d}.
		\end{equation*}
	\end{proof}
	
	Summing up, we have established the following main result (notice that for $ d = 2 $ we obtain the result from \cite{Pi-St} for all three functions):
	\begin{theorem}
		For every $ d \geq 2 $ the following identities hold and are best-possible:
		\begin{equation*}
			B_{f_{\min}}^d = \frac{1}{d}, \quad
			B_{f_{\max}}^d = \frac{d-1}{d}, \quad
			B_{f_{\sum}}^d = \frac{d^2-1}{3}
		\end{equation*}
	\end{theorem}	
	
	\begin{remark}
		As mentioned in the introduction, our main objective was to determine $ B_f^d $, given by
		\begin{equation*}
			B_f^d = \sup_{(x_n)_{n \in \NN} \in \mathcal{U}} \limsup\limits_{n \to \infty} \frac{1}{n} \sum_{i=1}^{n} f(x_i,x_{i+1},\ldots,x_{i+d-1})
		\end{equation*}
		for each of the three functions $ f_{\min}, f_{\max}, f_{\sum} $.
		Instead of tackling the problem directly, we first added degrees of freedom, considered
		\begin{equation*}
			C_f^d = \sup_{\xbf^1,\xbf^2,\ldots,\xbf^d \in \mathcal{U}} \limsup\limits_{n \to \infty} \frac{1}{n} \sum_{i=1}^{n} f(x_i^1,x_i^2,\ldots,x_i^d),
		\end{equation*}
		translated $ C_f^d $ to a maximization problem over $ \mathcal{P}_\textrm{ds}([0,1]^d) $, and solved the latter.
		The fact that the (sharp) upper bounds $ C_f^d $ obtained for the more general problem coincide with our bounds of interest $ B_f^d $ is to a certain extent surprising.
	\end{remark}

	\section{Outlook}
	
	In this paper we considered the functions $ f_{\min}, f_{\max}, f_{\sum} \colon [0,1]^d \to [0,\infty) $ considering all pairwise distances of its arguments.
	It seems equally interesting to study the related functions
	\begin{align*}
		g_{\min}(\xbf) &\coloneqq \min_{i = 1,\ldots,d-1} \abs{x_i-x_{i+1}}, \\
		g_{\max}(\xbf) &\coloneqq \max_{i = 1,\ldots,d-1} \abs{x_i-x_{i+1}}, \\
		g_{\sum}(\xbf) &\coloneqq \sum_{i = 1,\ldots,d-1} \abs{x_i-x_{i+1}},
	\end{align*}
	which only take into account/aggregate distances of consecutive arguments.
	Considering that changing to the just mentioned functions goes hand in hand with losing symmetry, we conjecture that deriving analogous results might be significantly more difficult.
	In fact, in this setting it is not enough any more to have a given pool of arguments but beyond this a non-cyclic (!) ordering is needed.
	In our language of black blocks, which started with highest and lowest one, and then was extended step by step by the number of layers, it would require some kind of timestamp.
	We leave it as an open problem.

\end{document}